\documentclass[11pt]{amsart}

\usepackage[a4paper,hmargin=3.5cm,vmargin=4cm]{geometry}
\usepackage{amsfonts,amssymb,amscd,amstext}
\usepackage{graphicx}
\usepackage[dvips]{epsfig}
\usepackage{pstricks,pst-plot}

\usepackage[utf8]{inputenc}
\usepackage{hyperref}

\usepackage{fancyhdr}
\input xy
\xyoption{all}

\usepackage{times}

\usepackage{enumerate}
\usepackage{titlesec}
\usepackage{mathrsfs}

\titleformat{\section}
{\filcenter\bfseries\large} {\thesection{.}}{0.2cm}{}
\titleformat{\subsection}[runin]
{\bfseries} {\thesubsection{.}}{0.15cm}{}[.]
\titleformat{\subsubsection}[runin]
{\em}{\thesubsubsection{.}}{0.15cm}{}[.]

\usepackage[up,bf]{caption}

\newtheorem{theorem}{Theorem}[section]
\newtheorem{proposition}[theorem]{Proposition}
\newtheorem{claim}[theorem]{Claim}
\newtheorem{lemma}[theorem]{Lemma}
\newtheorem{corollary}[theorem]{Corollary}

\theoremstyle{definition}
\newtheorem{definition}[theorem]{Definition}
\newtheorem{remark}[theorem]{Remark}

\newtheorem{problem}[theorem]{Problem}
\newtheorem{example}[theorem]{Example}

\numberwithin{equation}{section}
\numberwithin{figure}{section}

\newcommand\Lcal{\mathcal{L}}
\newcommand\Mcal{\mathcal{M}}
\newcommand\Ncal{\mathcal{N}}

\newcommand\Tcal{\mathcal{T}}

\newcommand\ba{\mathbf{a}}

\newcommand\bu{\mathbf{u}}
\newcommand\bv{\mathbf{v}}
\newcommand\bx{\mathbf{x}}
\newcommand\by{\mathbf{y}}
\newcommand\bw{\mathbf{w}}
\newcommand\bz{\mathbf{z}}

\newcommand\Cscr{\mathscr{C}}
\newcommand\Dscr{\mathscr{D}}

\renewcommand\b{\mathbb{B}}
\renewcommand\c{\mathbb{C}}
\newcommand\cd{\overline{\mathbb D}}

\renewcommand\d{\mathbb D}

\newcommand\n{\mathbb{N}}
\renewcommand\r{\mathbb{R}}
\newcommand\s{\mathbb{S}}
\renewcommand\t{\mathbb{T}}
\newcommand\z{\mathbb{Z}}

\newcommand\igot{\mathfrak{i}}

\renewcommand\igot{\mathfrak{i}}

\newcommand\E{\mathrm{e}}

\renewcommand\imath{\igot}

\newcommand\Agot{\mathfrak{A}}

\newcommand\Mgot{\mathfrak{M}}
\newcommand\Ngot{\mathfrak{N}}

\newcommand\zero{\mathbf{0}}

\newcommand\dist{\mathrm{dist}}
\renewcommand\span{\mathrm{span}}

\newcommand\wt{\widetilde}
\newcommand\wh{\widehat}
\newcommand\di{\partial}

\newcommand\Flux{\mathrm{Flux}}

\newcommand\Psh{\mathrm{Psh}}
\newcommand\MPsh{\mathrm{\mathfrak{M}Psh}}

\newcommand\NPsh{\mathrm{\mathfrak{N}Psh}}

\newcommand\Co{\mathrm{Co}}
\newcommand\supp{\mathrm{supp}}

\newcommand\tr{\mathrm{tr}}
\newcommand\Hess{\mathrm{Hess}}

\usepackage{color}

\begin{document}

\fancyhead[LO]{Minimal surfaces in minimally convex domains}
\fancyhead[RE]{A.\ Alarc\'on, B.\ Drinovec Drnov\v sek, F.\ Forstneri\v c, and F.\ J.\ L\'opez}
\fancyhead[RO,LE]{\thepage}

\thispagestyle{empty}

\vspace*{7mm}
\begin{center}
{\bf \LARGE Minimal surfaces in minimally convex domains}
\vspace*{5mm}

{\large\bf A.\ Alarc\'on, B.\ Drinovec Drnov\v sek, F.\ Forstneri\v c, and  F.\ J.\  L\'opez}

\vspace*{5mm}
{\em Dedicated to Josip Globevnik for his seventieth birthday}
\end{center}

\vspace*{6mm}

\begin{quote}
{\small
\noindent {\bf Abstract}\hspace*{0.1cm}
In this paper, we prove that every conformal minimal immersion of a compact bordered 
Riemann surface $M$ into a minimally convex domain $D\subset \r^3$ can be approximated, 
uniformly on compacts in $\mathring M=M\setminus bM$, 
by proper complete conformal minimal immersions $\mathring M\to D$ (see Theorems \ref{th:main1},
\ref{th:main1-bis}, and \ref{th:main2}). 
We also obtain a rigidity theorem for complete immersed minimal surfaces  of finite total curvature 
contained in a minimally convex domain in $\r^3$
(see Theorem \ref{th:FTC}), and we  characterize the minimal surface hull of a 
compact set $K$ in $\r^n$ for any $n\ge 3$ by  sequences of conformal minimal discs 
whose boundaries converge to $K$ in the measure theoretic sense
(see Corollary \ref{cor:minullhull}).

\vspace*{0.1cm}
\noindent{\bf Keywords}\hspace*{0.1cm} Riemann surface, minimal surface, minimally convex domain.

\vspace*{0.1cm}

\noindent{\bf MSC (2010):}\hspace*{0.1cm} 53A10; 32B15, 32E30, 32H02.}
\end{quote}

%
%
%
%
\section{Introduction}\label{sec:intro}
A major problem in minimal surface theory is to understand
which domains in $\r^3$ admit complete properly immersed minimal surfaces, 
and how the geometry of the domain influences the conformal properties of such surfaces.
(For background  on this topic,  see e.g.\ \cite[Section 3]{MeeksPerez2004SDG}.) 
In the present paper, we obtain general existence and approximation results for complete proper 
conformal minimal immersions  
from an arbitrary bordered Riemann surface  into any minimally convex domain in $\r^3$; see 
Theorems \ref{th:main1}, \ref{th:main1-bis} and \ref{th:main2}. We also  
show that one cannot expect similar results in a wider class of domains in $\r^3$.

Let $n\ge 3$. A domain  $D\subset\r^n$ is said to be {\em minimally convex} if it admits
a smooth exhaustion function $\rho\colon D\to\r$ that is {\em strongly $2$-plurisubharmonic}
(also called {\em minimal strongly plurisubharmonic}), 
meaning that for every point $\bx\in D$, the sum of the smallest two eigenvalues of the Hessian $\Hess_\rho(\bx)$ is positive.
(See Definitions \ref{def:p-psh} and \ref{def:sp-psh}.)
A domain $D$ with $\Cscr^2$ boundary is minimally convex if and only if $\kappa_1(\bx)+\kappa_2(\bx)\ge 0$ 
for each point $\bx\in bD$, where $\kappa_1(\bx)\le \kappa_2(\bx)\le \cdots\le \kappa_{n-1}(\bx)$ 
are the normal curvatures of $bD$ at the point $\bx$ with respect to the inner normal (see Theorem \ref{th:p-convex}).
In particular, a domain in $\r^3$ bounded by a properly embedded minimal surface
is minimally convex (see Corollary \ref{cor:compl-minimal}). 
Clearly, every convex domain is also minimally convex, but  there exist minimally convex domains without any convex boundary points
(see Example \ref{ex:catenoid}).

Our first main result is the following.

%
%
\begin{theorem}\label{th:main1}
Assume that $D$ is a minimally  convex domain in $\r^3$, and let $M$ be a compact bordered Riemann surface
with nonempty boundary $bM$.  Then, every conformal minimal immersion $F_0\colon M\to D$  
can be approximated, uniformly on compacts in $\mathring  M=M\setminus bM$, by proper complete conformal 
minimal immersions $F\colon \mathring M\to D$ with $\Flux(F)=\Flux(F_0)$. 
\end{theorem}

Recall that a {\em compact bordered Riemann surface} is a compact connected oriented surface, 
$M$, endowed with a complex structure, whose boundary $bM\neq\emptyset$ consists of finitely many 
smooth Jordan curves.  The interior, $\mathring M=M\setminus bM$, of such $M$ is an (open) 
{\em bordered Riemann surface}. A {\em conformal minimal immersion} $F\colon M\to\r^n$ is an immersion 
which is angle preserving and harmonic; such a map parametrizes a minimal surface in $\r^n$. 
The {\em flux} of $F$ is the group homomorphism $\Flux (F)\colon H_1(M,\z)\to \r^n$ whose value on any closed 
oriented curve $\gamma\subset M$ is $\Flux(F)(\gamma)=\oint_\gamma \Im (\partial F)$;
here, $\partial F$ is the $(1,0)$-differential of $F$ and $\Im$ denotes the imaginary part.
An immersion $F\colon \mathring M\to \r^n$  is said to be {\em complete} if the pull-back $F^*ds^2$ 
of the Euclidean metric on $\r^n$ is a complete Riemannian metric  on $ \mathring M$. 

Note that Theorem \ref{th:main1} pertains to a fixed conformal structure on the surface $M$. 
The analogous result for {\em convex domains} in $\r^n$ for any $n\ge 3$ 
is \cite[Theorem 1.4]{AlarconDrinovecForstnericLopez2015}; see also \cite[Theorem 1]{AlarconLopez2013MA}.
Theorem \ref{th:main1} seems to be the first general existence and approximation result for 
(complete) proper minimal surfaces in a class of domains in $\r^3$ which contains all convex domains, 
but also many non-convex ones; convexity has been impossible to avoid with the existing construction methods. 
Comparing with the results in the literature, 
it is known that there are properly immersed minimal  surfaces in $\r^3$ with arbitrary conformal structure 
(see \cite{AlarconLopez2012JDG,AlarconLopez2014TAMS,AlarconLopez2015GT}), 
and that every domain $D\subset\r^3$ which is convex, or has a smooth strictly convex boundary point, 
admits complete properly immersed minimal surfaces that 
are conformally equivalent to any given bordered Riemann surface 
(see \cite{AlarconDrinovecForstnericLopez2015}). These were the most general known results in this direction up to now.

As shown by Remark \ref{rem:precise} and Examples \ref{ex:obstacles} and \ref{ex:no-disc}, 
the hypothesis of minimal convexity is essentially optimal in Theorem \ref{th:main1}. 
In Example \ref{ex:obstacles} we exhibit a bounded, simply connected domain $D\subset \r^3$
such that a certain conformal minimal disc $F_0\colon \overline\d\to D$ cannot be approximated by {\em proper} 
conformal minimal discs $\d\to D$.  (Here, $\d=\{\zeta\in \c: |\zeta|<1\}$.)
In another direction, Mart\'in, Meeks and Nadirashvili constructed bounded (non-simply connected) domains in $\r^3$ 
which do not admit any complete properly immersed minimal surfaces with an annular end (see \cite{MartinMeeksNadirashvili2007AJM}). 
We point out in Example \ref{ex:no-disc} that there is a domain from \cite{MartinMeeksNadirashvili2007AJM} which
does not admit any proper minimal discs. Clearly, Theorem \ref{th:main1} fails in both these examples 
even without the completeness condition.

In Remark \ref{rem:generalizations}, we indicate a generalization of
Theorem \ref{th:main1}, and of the related subsequent results in this paper, to a certain class of
not necessarily convex domains in $\r^n$ for any $n>3$.
However, we have optimal results only in dimension $n=3$.

Theorem \ref{th:main1} is proved in Section  \ref{sec:Proof1}; here is a brief outline. 
Let $\rho\colon D\to\r$ be a Morse minimal strongly plurisubharmonic exhaustion function
with the (discrete) critical locus $P$.  For every point $\bx\in D\setminus P$ we find a small
embedded conformal minimal disc $\bx \in \Mcal_\bx \subset D$ such that 
the restriction of $\rho$ to $\Mcal_\bx$ has a strict minimum at $\bx$ and increases quadratically.
Furthermore, for points in a simply connected compact set in $D\setminus P$ we can choose a smooth family
of such discs satisfying uniform estimates for the rate of growth of $\rho$ (see Lemma \ref{lem:M-discs}). 
By using these discs and an approximate solution of a Riemann-Hilbert type boundary value problem
(see Theorem \ref{th:RH}), we can lift the boundary of a given conformal minimal immersion $M\to D$  to a higher level of the 
function $\rho$, paying attention not to decrease the level of $\rho$ much anywhere on $M$ and to approximate 
the given immersion on a chosen compact subset of $\mathring M$ (see Proposition \ref{prop:lifting}).  
This procedure can be carried out so that the image of the boundary $bM$ avoids the critical locus 
of $\rho$. A recursive application of this lifting method leads to the construction
of a proper conformal minimal immersion $\mathring M\to D$. 
(Analogous results for proper holomorphic maps can be found 
in \cite{DrinovecForstneric2007DMJ,DrinovecForstneric2010AJM}.) 
This construction method is geometrically  simpler 
than the one developed by the authors in \cite{AlarconDrinovecForstnericLopez2015},
the main advantage being the higher flexibility of the Riemann-Hilbert method that is available 
in dimension $n=3$ (compare Theorem \ref{th:RH} with \cite[Theorem 3.5]{AlarconDrinovecForstnericLopez2015}).

Completeness of the immersion is achieved by combining the boundary lifting procedure
with a technique, developed recently in \cite{AlarconDrinovecForstnericLopez2015}, 
that enables one to increase the intrinsic boundary distance 
in $M$ by an arbitrarily big amount while staying $\Cscr^0$ close to a given 
conformal minimal immersion  $M\to \r^n$ (see \cite[Lemmas 4.1 and 4.2]{AlarconDrinovecForstnericLopez2015}).  
A recursive application of these two techniques yields 
Theorem \ref{th:main1}. (See Section  \ref{sec:Proof1} for the details.)

Before proceeding, we place the class of minimal plurisubharmonic functions and 
minimally convex domains into a wider framework, and we provide some examples.

Minimal plurisubharmonic functions are a special case, with $p=2$, of the class of 
{\em $p$-plurisubharmonic functions}
which have been studied by Harvey and Lawson in  \cite{HarveyLawson2013IUMJ};
see also \cite{DrinovecForstneric2015TAMS,HarveyLawson20092AJM,HarveyLawson2011ALM,HarveyLawson2012AM}.
A real-valued $\Cscr^2$  function $u$ on a domain $D\subset \r^n$ is said to be (strongly) 
$p$-plurisubharmonic for some integer $p\in  \{1,2,\ldots,n\}$ if the restriction of $u$ 
to any $p$-dimensional affine subspace of $\r^n$ is (strongly) subharmonic (see Definition \ref{def:p-psh});
equivalently, if the sum of the $p$ smallest eigenvalues of the Hessian of $u$ is 
nonnegative (positive) at every point (see Proposition \ref{prop:characterizations}).
The restriction of a $p$-plurisubharmonic function to a $p$-dimensional
minimal submanifold is a subharmonic function on the submanifold (see Proposition \ref{prop:characterizations}).
Note that $1$-plurisubharmonic functions are convex functions,
while $n$-plurisubharmonic  functions are subharmonic functions.
The set $\Psh_p(D)$ of all $p$-plurisubharmonic functions on $D$ 
is closed under addition and multiplication by nonnegative numbers.

A domain $D\subset\r^n$ is said to be {\em $p$-convex} if it admits a strongly $p$-plurisubharmonic exhaustion function 
$\rho\colon D\to\r$ (see Definition \ref{def:sp-psh} and Proposition \ref{prop:p-convex}).
Thus, $1$-convex domains are linearly convex, while $2$-convex domains are minimally convex. 
Every domain in $\r^n$ is $n$-convex; this is a special case of a theorem of 
Greene and Wu \cite{GreeneWu1975} (see also Demailly \cite{Demailly1990})
that every connected noncompact Riemannian manifold 
admits  a smooth strongly subharmonic exhaustion function.
Harvey and Lawson proved that, for smoothly bounded domains in $\r^n$, 
$p$-convexity is a local property of the boundary, akin to Levi pseudoconvexity in complex analysis. 
For future reference, we state the following summary of their main results  from \cite{HarveyLawson2013IUMJ}.
Harvey and Lawson considered bounded domains in $\r^n$, but we show  
in  Section \ref{ss:strongly-p-convex} that Theorem \ref{th:p-convex} also holds for unbounded domains.

%
%
\begin{theorem}[Section 3 in \cite{HarveyLawson2013IUMJ}] 
\label{th:p-convex}
Let $1\le p<n$ be integers, and let $D\subset \r^n$ be a 
domain  with $\Cscr^2$ boundary, not necessarily bounded. 
The following conditions are equivalent.
\begin{itemize} 
\item[\rm (a)]  $D$ is $p$-convex. 
\item[\rm (b)]   There exist a neighborhood $U\subset \r^n$ of $bD$ and a $\Cscr^2$ function $\rho\colon U\to \r$ 
such that  $D\cap U=\{\rho<0\}$,  $d\rho\ne 0$ on $bD\cap U=\{\rho=0\}$, and 
\begin{equation}\label{eq:trace-bD}
	\tr_L \Hess_\rho(\bx) \ge 0 \ \ \text{for every tangent $p$-plane}\ L\subset T_\bx bD,\ 
	\bx\in bD.
\end{equation}
(Here $\Hess_\rho(\bx)$ is the Hessian (\ref{eq:Hess})  of $\rho$ at $\bx$ and $\tr_L$ denotes the trace of the 
restriction to $L$.) Property (\ref{eq:trace-bD})  is  independent of the choice of $\rho$. 
\item[\rm (c)] 
If $\bx\in bD$ and $\kappa_1\le \kappa_2\le\ldots\le \kappa_{n-1}$ are the principal curvatures of $bD$
from the inner side at $\bx$, then
$ 
 	\kappa_1 +  \kappa_2 + \cdots + \kappa_{p} \ge 0.
$
\item[\rm (d)]  There exists a neighborhood $U$ of $bD$ such that the function $-\log \dist(\cdotp,bD)$ 
is $p$-plurisubharmonic  on $D\cap U$.
\end{itemize}
\end{theorem}

Theorem \ref{th:p-convex} shows in particular that a domain $D\subset\r^3$ with $\Cscr^2$
boundary is minimally convex if and only if the principal curvatures of the boundary $bD$
satisfy $\kappa_1(\bx)+\kappa_2(\bx)\ge 0$ at every point $\bx\in bD$. Theorem \ref{th:main1}
applies to any such domain.

The following is a  corollary to Theorem \ref{th:p-convex} in the case  $D=\r^3$
(note that we have $\kappa_1 +  \kappa_2=0$ on a minimal surface $S\subset \r^3$);
the general case is proved in Section \ref{ss:p-convex}.

%
%
\begin{corollary} \label{cor:compl-minimal}
If $S$ is a properly embedded minimal surface in $\r^3$, then every connected
component of $\r^3\setminus S$ is a minimally convex domain. More generally, 
if $D\subset \r^3$ is a minimally convex domain and $S$ is a closed embedded minimal 
surface in a neighborhood of $\overline D$, then every connected component 
of $D\setminus S$ is minimally convex. 
\end{corollary}

%
%
\begin{example}\label{ex:catenoid}
Let $D$ be the domain
\[
	D=\{(x,y,z)\in\r^3: x^2+y^2>\cosh^2 z\}.
\] 
Since the boundary of $D$ is a minimal surface (a catenoid), $D$ is minimally convex
by Corollary \ref{cor:compl-minimal}. Clearly, $D$  does not have any convex boundary point,
and its fundamental group $\pi_1(D)$ equals $\z$.
\qed\end{example}

\begin{remark}\label{rem:homotopy}
Note that the Hessian of a minimal strongly plurisubharmonic function on a domain in $\r^3$ has at most one negative eigenvalue 
at every point. Hence, Morse theory implies that a minimally convex domain $D$ has the homotopy type of a 
$1$-dimensional CW complex; in particular, the higher homotopy groups $\pi_k(D)$ for $k>1$ all vanish. 
Similarly, a $p$-convex domain has the homotopy type of a CW complex of dimension at most $p-1$.  
\qed \end{remark}

%
%
\begin{remark}
\label{rem:mean-convex}
In the literature on minimal surfaces, a smoothly bounded domain $D$ in $\r^n$ 
is said to be (strongly) {\em mean-convex} if the sum of the principal curvatures of $bD$ 
from the interior side is nonnegative (resp.\ positive) at each point.  
This is precisely condition (c) in Theorem \ref{th:p-convex} with $p=n-1$; hence, a smoothly 
bounded domain in $\r^n$ is mean-convex if and only if it is $(n-1)$-convex.
In particular, mean-convex domains in $\r^3$ coincide with smoothly bounded  minimally convex domains. 
Mean-convex domains have been studied as natural barriers for minimal hypersurfaces 
in view of the maximum principle; see  Section \ref{ss:maximum} and Remark \ref{rem:mean-hull}.
Nontrivial proper minimal hypersurfaces in mean-convex domains often arise as solutions to Plateau problems. 
For instance, Meeks and Yau \cite{MeeksYau1982MZ} proved that every null-homotopic Jordan curve  in the 
boundary of a mean-convex domain $D\subset \r^3$ bounds an area minimizing minimal disc in $D$. 
This method does not seem to provide examples of complete minimal surfaces, 
or those normalized by a given bordered Riemann surface other than the disc. 
For a discussion of this subject, see e.g.\ \cite[Section 6.5]{ColdingMinicozzibook}.
\qed \end{remark}

Our proof of Theorem \ref{th:main1} also shows that boundaries of conformal minimal surfaces can be 
pushed to a {\em minimally convex end} of a domain $D\subset \r^3$  as in the following theorem.  
An analogous result in the holomorphic category is  \cite[Theorem 1.1]{DrinovecForstneric2010AJM}.

%
%
%
%
\begin{theorem} \label{th:main1-bis}
Assume that $\Omega\subset D$ are open sets in $\r^3$ and
$\rho\colon \Omega\to (0,+\infty)$ is a smooth minimal strongly plurisubharmonic function
such that for any pair of numbers $0<c_1<c_2$ the set 
$
	\Omega_{c_1,c_2}= \{\bx \in\Omega\colon c_1\le \rho(\bx)\le c_2\}
$ 
is compact. Let $M$ be a compact bordered Riemann surface with nonempty boundary $bM$.  
Every conformal minimal immersion $F_0\colon M\to D$ satisfying $F_0(bM)\subset \Omega$
can be approximated, uniformly on compacts in $\mathring  M=M\setminus bM$, 
by complete conformal minimal immersions $F\colon \mathring M\to D$
such that  $F(z)\in \Omega$ for every $z\in \mathring M$ sufficiently close to $bM$ and 
\begin{equation}\label{eq:infinity}
	\lim_{z\to bM} \rho(F(z))=+\infty.
\end{equation}
\end{theorem}

In a typical application of Theorem \ref{th:main1-bis}, 
the set $\Omega$ is a collar around a minimally convex boundary component $S\subset bD$.
(By Theorem \ref{th:p-convex}, a smooth boundary component $S\subset bD$ is minimally convex 
if and only if $-\log \dist(\cdotp,S)$ is minimal plurisubharmonic near $S$.) 
Theorem \ref{th:main1-bis} furnishes a proper complete conformal minimal immersion 
$F\colon \mathring M\to D$ whose 
boundary cluster set is contained in $S$ as shown by condition (\ref{eq:infinity}).

%
%

Next, we consider the class of strongly minimally convex domains.

\begin{definition} \label{def:SMC}
A domain  $D\subset \r^n$ with $\Cscr^2$ boundary is {\em strongly $p$-convex} 
for some $p\in \{1,\ldots,n-1\}$
if it admits a $\Cscr^2$ defining function $\rho$ on a neighborhood $U$ of $bD$
(i.e., $D\cap U=\{\rho<0\}$  and $d\rho\ne 0$ on $bD=\{\bx\in U\colon \rho(\bx)=0\}$)
whose Hessian satisfies the strict inequality in (\ref{eq:trace-bD}):
\[	
	\tr_L \Hess_\rho(\bx) > 0 \ \ \text{for every tangent $p$-plane}\ L\subset T_\bx bD,\ 
	\bx\in bD.
\]
A strongly $2$-convex domain is said to be {\em strongly minimally convex}.
\end{definition}

The analogue of Theorem \ref{th:p-convex} holds in this setting. In particular,
a bounded domain $D\Subset\r^n$ with $\Cscr^2$ boundary is strongly $p$-convex for some $p\in \{1,\ldots,n-1\}$
if and only if the principal curvatures $\kappa_1\le \kappa_2\le\ldots\le \kappa_{n-1}$ 
of $bD$ at any point $\bx\in bD$ satisfy
$
 	\kappa_1 +  \kappa_2 + \cdots + \kappa_{p} > 0.
$
Note that $D$ is strongly $(n-1)$-convex if and only if it is strongly mean-convex 
(see Remark \ref{rem:mean-convex}).

Our next result improves Theorem \ref{th:main1} for bounded strongly minimally convex domains.

%
%
\begin{theorem}\label{th:main2}
Let $D$ be a bounded strongly minimally convex domain in $\r^3$ (Definition \ref{def:SMC}).
Given a compact bordered Riemann surface $M$ with nonempty boundary $bM$ and
a conformal minimal immersion $F_0\colon M\to D$, we can approximate $F_0$,   
uniformly on compacts in $\mathring  M$, by continuous maps 
$F\colon M\to \overline D$ such that $F(bM)\subset bD$, $F\colon \mathring M\to D$ is a
proper complete conformal minimal immersion, $\Flux(F_0)=\Flux(F)$, and 
\begin{equation}\label{eq:root}
	 \sup_{\zeta\in M} \|F(\zeta)-F_0(\zeta)\|  \le C\sqrt{\max_{\zeta\in bM} \dist(F_0(\zeta),bD)}
\end{equation}
for some constant $C>0$ depending only on $D$.
\end{theorem}

The improvement over Theorem \ref{th:main1} is that the approximating map $F$ can now be chosen continuous up to the 
boundary of $M$, so $F(bM)\subset bD$ is a union of finitely many closed curves, and we have the estimate 
\eqref{eq:root}. Since $F$ is complete, the minimal surface $F(\mathring M)\subset D$ 
has infinite area, and hence its boundary $F(bM)$ is necessarily non-rectifiable in view of the isoperimetric inequality. 
The corresponding result for smoothly bounded strongly convex domains in $\r^n$ for any $n\ge 3$ is 
\cite[Theorem 1.2]{AlarconDrinovecForstnericLopez2015};
see also \cite{Alarcon2010AIP} for a previous partial result in this line.  As in the latter result,  we are unable to 
achieve that $F$ be a topological embedding on $bM$, so $F(bM)$ needs not consist of Jordan curves.

Theorem \ref{th:main2} implies the following corollary.

\begin{corollary}\label{cor:str-convex}
Every domain $D\subset \r^3$ having a $\Cscr^2$ strongly minimally convex boundary point contains complete properly 
immersed minimal surfaces extending continuously up to the boundary and normalized by any given bordered 
Riemann surface. 
\end{corollary}

\begin{proof} Assume that $\bx_0\in bD$ is a strongly minimally convex boundary point, i.e., 
such that $\kappa_1(\bx_0)+\kappa_2(\bx_0)>0$.
Then there are a neighborhood $U$ of $\bx_0$ and a strongly minimally convex domain $D'\subset D$
such that $D\cap U=D'\cap U$. (It suffices to intersect $D$ by a small ball around $\bx_0$ and smooth the corners.) 
Given a conformal minimal immersion $F_0\colon M\to D'$ whose image
$F_0(M)$ lies close enough to the point $\bx_0$ (such exists since $M$ is compact), 
the map $F\colon M\to \overline D'$, furnished by Theorem \ref{th:main2},  satisfies $F(bM)\subset bD\cap U$ in view of 
the estimate (\ref{eq:root}), and hence the map $F\colon \mathring M\to D$ is proper.
\end{proof}

Following Meeks and P\'{e}rez \cite[Section 3]{MeeksPerez2004SDG}, a domain $W\subset\r^3$ is said to be 
{\em universal for minimal surfaces} if every complete, connected, properly immersed minimal surface in $W$ is either 
recurrent (when the surface is open), or a parabolic surface with boundary. Since every open bordered Riemann surface 
$\mathring M=M\setminus bM$ is transient,
%
%
%
%
Theorem  \ref{th:main1-bis} and Corollary \ref{cor:str-convex} show that every domain $D\subset\r^3$ 
which has a minimally convex end, or a strongly minimally convex boundary point, fails to be universal for minimal surfaces. 
In particular, there are domains in $\r^3$ which are not universal for minimal surfaces and have no convex boundary points;
for example, the catenoidal domain in Example \ref{ex:catenoid}.

%
%
\begin{remark}\label{rem:precise}
The conclusion of Theorem \ref{th:main1-bis} fails along a compact smooth boundary component 
$S \subset bD$ which is {\em strongly minimally concave}, i.e., $\kappa_1(\bx)+\kappa_2(\bx)<0$ for every point $\bx\in S$. 
Indeed, Theorem \ref{th:p-convex} furnishes an open neighborhood $U$ of $S$ in $\r^3$ 
and a minimal strongly plurisubharmonic function $\phi\colon U\to\r$ that vanishes on $S$
and is positive on $D\cap U$. The maximum principle applied with $\phi$ 
shows that there is no minimal surface in $\overline D\cap U$ with boundary in $S$.
The same argument holds locally near  a smooth strongly minimally concave boundary point
$\bx_0\in bD$; in this case there is a neighborhood $U\subset \r^3$ of $\bx_0$ 
such that there are no proper minimal surfaces in $D\cap U$ with boundary in $bD\cap U$. 

For a {\em complete} proper minimal surface there is another restriction on the location of
its boundary points. Assume that $D \subset\r^3$ is a domain with $\Cscr^2$ boundary and $F \colon \d\to D$ 
is a complete conformal proper minimal immersion from the disc $\d=\{\zeta\in \c\colon |\zeta|< 1\}$
extending continuously to $\overline\d$.
Then the boundary curve $F(b\d)\subset bD$ does not contain any strongly concave boundary 
points of $D$ (see \cite{AlarconNadirashvili2008MZ}).
This is especially relevant in connection to Theorem  \ref{th:main2}. 
However, we do not know whether $F(b\d)$ could contain a strongly minimally concave
boundary point; the following remains an open problem.

\begin{problem}\label{prob:optimal}
Let $D$ be a smoothly bounded domain in $\r^3$ and $F \colon \d \to D$ 
be a complete conformal proper minimal immersion extending continuously to $\overline\d$
(hence $F(b\d)\subset bD$). Do we have $\kappa_1(\bx)+\kappa_2(\bx)\geq 0$ for every point $\bx\in F(b\d)$?
\qed\end{problem}
\end{remark}

We now illustrate by a couple of examples that Theorem \ref{th:main1} fails in general
for domains in $\r^3$ which are not minimally convex. Since every  domain in $\r^3$ is 3-convex 
(i.e., it admits a strongly subharmonic exhaustion function see \cite{GreeneWu1975,Demailly1990}),
we see in particular that the hypothesis of $2$-convexity cannot be replaced by 3-convexity
in Theorem \ref{th:main1}.

\begin{example}\label{ex:obstacles}
We exhibit a simply connected domain $D\subset \r^3$ of the form $D=2\b\setminus K$, where 
$\b$ is the unit ball of $\r^3$ and $K$ is a compact set contained in a thin shell
around the unit sphere $\s=b\b$, such that the image of every proper conformal minimal
disc $F\colon\d\to D$ avoids the ball $\frac{1}{2}\b\subset D$. 
Clearly, Theorem \ref{th:main1} fails in this example even without the completeness condition. 
Note however that $D$ admits complete properly immersed minimal surfaces normalized by any given bordered 
Riemann surface in view of Theorem \ref{th:main1-bis},
applied to the strongly convex boundary component $2\s\subset bD$.

The example is essentially the one given in \cite[Section 5]{ForstnericGlobevnik1992}
in the context of holomorphic discs in domains in $\c^2$. 
We cover the unit sphere $\s\subset \r^3$ by small open spherical caps $C_1,\ldots, C_m$ 
(i.e., every $C_j$ is the intersection  of $\s$ by a half-space defined by an affine plane $H_j \subset\r^3$) 
such that $\bigcup_{j=1}^m \Co(\overline  C_j) \cap \frac{1}{2}\overline\b=\emptyset$. 
(Here, $\Co$ denotes the convex hull.) 
Pick a number $r>1$ so close to $1$ that  $\s\subset \bigcup_{j=1}^m \Co(r C_j)$. 
Choose pairwise distinct numbers $\rho_1,\ldots,\rho_m$ very close to $r$
such that the pairwise disjoint spherical  caps $\Gamma_j=\rho_j C_j$ satisfy 
$\s\subset \bigcup_{j=1}^m \Co(\Gamma_j)$ and $\bigcup_{j=1}^m \Co(\overline \Gamma_j) \cap \frac{1}{2}\overline\b=\emptyset$.
Let $D=2\b\setminus \bigcup_{j=1}^m \overline \Gamma_j$. For any 
proper conformal minimal disc $F\colon \d\to D$, its boundary cluster set
$\Lambda(F)$  (i.e., the set of all limit points $\lim_{j\to\infty} F(\zeta_j)\in bD$ along  sequences $\zeta_j\in \d$
with $\lim_{j\to\infty} |\zeta_j|=1$) is a connected compact set in $bD$; hence it is contained 
in the sphere $2\s$ or in one of the caps $\Gamma_j$. 
Assume now that $F(\zeta_0)\in \frac{1}{2}\overline\b$ for some $\zeta_0\in \d$. 
If $\Lambda(F)\subset \Gamma_j$ for some $j\in \{1,\ldots,m\}$,  then $F(\d)\subset \Co(\Gamma_j)$
by the maximum principle, a contradiction since $\Co(\Gamma_j)$ does not intersect the ball 
$\frac{1}{2}\overline\b$. If on the other hand $\Lambda(F)\subset 2\s$,  there is 
a point $\zeta_1\in\d$ with $F(\zeta_1)\in \s$. Pick $j\in \{1,\ldots,m\}$ such that $F(\zeta_1) \in \Co(\Gamma_j)$.
Since $F$ has no cluster points on $\Gamma_j$, the set 
$U =\{\zeta\in\d\colon F(\zeta)\in \Co(\Gamma_j)\}$ 
is a nonempty relatively compact domain in $\d$, and $F(bU)$ lies in the affine plane $H_j$ 
which determines the spherical cap $\Gamma_j$. 
By the maximum principle it follows that $F$ maps all of $U$, and hence the whole disc $\d$,
into $H_j$, a contradiction. 
\qed\end{example}

Mart\'in, Meeks and Nadirashvili  constructed bounded domains in $\r^3$ which do  not admit any proper 
complete minimal surfaces of finite topology (see \cite{MartinMeeksNadirashvili2007AJM}). 
In the next example we show that the collection in \cite{MartinMeeksNadirashvili2007AJM} includes a domain $D\subset\r^3$
which carries no proper minimal discs, irrespectively of completeness. 
A similar result in the holomorphic category is due to Dor \cite{Dor1996} who 
constructed a bounded domain $D$ in $\c^m$ for any $m\ge 2$ which does not
admit any proper holomorphic discs.

\begin{example}\label{ex:no-disc}
Let $S$ be the cylindrical shell 
\[
	S=\left\{(x,y,z)\in\r^3 : 1<\|(x,y)\|=\sqrt{x^2+y^2}<2,\; 0<z<1\right\}.
\]
For $0<t<1$, let $C_t:=S\cap\{(x,y,z)\in\r^3\colon z=t\}$ denote the planar round open annulus
obtained by intersecting the cylinder $S$ with the plane $z=t$. For $j\in\n$, 
denote by $C_{t,j}$ the planar round compact annulus $C_{t,j}=\{(x,y,z)\in C_t\colon 1+\frac1{2j}\le \|(x,y)\| \leq 2-\frac1{2j}\}$. Obviously, $bC_{t,j}=\{(x,y,z)\in \r^3\colon \|(x,y)\|\in\{1+\frac1{2j},2-\frac1{2j}\},\, z=t\}$.  Let $t_1,t_2,t_3,\ldots$ denote the sequence 
\[
	\frac12 \,,\, \frac13 \,,\, \frac23 \,,\, \frac14 \,,\, \frac24 \,,\, \frac34 \,,\, 
	\cdots \,,\, \frac1{n} \,,\, \frac2{n} \,,\, \cdots \,,\, \frac{n-1}{n} \,,\, \cdots.
\]
Set $\Gamma=\bigcup_{j\in\n} bC_{t_j,j}\subset S$ and $D=S\setminus \Gamma$.
By \cite[proof of Theorem 1]{MartinMeeksNadirashvili2007AJM}, $D$ is a domain in $\r^3$ and the boundary cluster set $\Lambda(E)\subset \overline D\setminus D$ of any proper minimal annular end $E\subset D$ lies in a horizontal plane of $\r^3$. By the maximum principle, this implies that every proper minimal disc $\d\to D$ is contained in a horizontal plane, but clearly $D$ does not admit any such discs. More generally, $D$ does not carry any proper minimal surfaces of finite genus and with a single end.
\qed\end{example}

All minimal surfaces in Theorems \ref{th:main1}, \ref{th:main1-bis}, and \ref{th:main2} 
are images of bordered Riemann surfaces, hence finitely connected. 
If one does not insist on the approximation and the control of the conformal structure
on these surfaces, then our methods also give complete proper minimal surfaces of 
arbitrary topological type.

\begin{corollary}\label{co:inf-top}
If $D$ is a domain in $\r^3$ which has a minimally convex end in the sense of Theorem  \ref{th:main1-bis},
or a strongly minimally convex boundary point, then every open orientable smooth surface $S$ carries a 
complete proper minimal immersion $S\to D$ with arbitrary flux.
\end{corollary}

Corollary \ref{co:inf-top} is proved at the end of Section \ref{sec:Proof1}.
For domains $D\subset\r^n$ $(n\geq 3)$ that are convex, or have a $\Cscr^2$ smooth strictly convex boundary point, 
this has already been established in \cite{AlarconDrinovecForstnericLopez2015};
for $n=3$ see also Ferrer, Mart\'in, and Meeks \cite{FerrerMartinMeeks2012AM}.

Theorems \ref{th:main1}, \ref{th:main1-bis} and \ref{th:main2} show that every minimally convex domain in $\r^3$ admits 
many complete properly immersed minimal surfaces of {\em hyperbolic} conformal type. In contrast, the following 
rigidity type result shows that only very few minimally convex domains contain  a complete proper minimal  
surface $S\subset \r^3$ of {\em finite total curvature}. 
(Note that these are the simplest complete minimal surfaces of {\em parabolic} conformal type.)

%
%
\begin{theorem} \label{th:FTC}
Let $S\subset \r^3$ be a complete connected properly immersed minimal surface with finite total curvature
in $\r^3$. If $D\subset\r^3$ is a smoothly bounded minimally convex domain containing $S$, then $D=\r^3$ or $S$ is a plane;
in the latter case, the connected component of $D$ containing $S$ is a slab, a halfspace, or $\r^3$.
\end{theorem}

By a {\em slab} in $\r^3$, we mean a domain bounded by two parallel planes.

In particular, if $D$ is a connected component of $\r^3\setminus S$ where
$S$ is a non-flat properly embedded minimal surface of finite total curvature in $\r^3$, 
then Theorem \ref{th:FTC} shows that $D$ is a {\em  maximal minimally convex domain}, 
in the sense that the only smoothly bounded minimally convex domain containing 
$\overline D$ is $\r^3$ itself.

Theorem \ref{th:FTC} is proved in Section \ref{sec:maximal} as an application of a general maximum principle at infinity for complete, finite total curvature, noncompact minimal surfaces with compact boundary in minimally convex domains of $\r^3$; see Theorem \ref{th:FTC-infinity}. Maximum principles at infinity have been the key in many celebrated classification results in the theory of minimal surfaces; see for instance \cite{MeeksRosenberg2008JDG} and the references therein.  In the proof of Theorem \ref{th:FTC-infinity}, we exploit the geometry of complete minimal surfaces of finite total curvature along with the Kontinuit\"atssatz for conformal minimal surfaces; see Proposition \ref{prop:Kontinuitatssatz} for the latter.

In Section \ref{sec:hulls}, we indicate how the Riemann-Hilbert technique, developed in \cite{AlarconDrinovecForstnericLopez2015},
allows us to extend all main results of the paper \cite{DrinovecForstneric2015TAMS}  to null hulls of compact sets in $\c^n$ 
(see Definition \ref{def:nullhull})  and minimal hulls of compact sets in $\r^n$ (see Definition \ref{def:p-hull}) for any $n\ge 3$. 

After the completion of this paper, Alarc\'on, Forstneri\v c, and L\'opez obtained analogues of 
Theorems \ref{th:main1} and \ref{th:main2} in the non-orientable framework (see \cite{AlarconForstnericLopez2016MAMS}).

%
%

\section{$p$-plurisubharmonic functions and $p$-convex domains}
\label{sec:pPsh}

We begin this preparatory section by summarizing basic results concerning $p$-pluri\-subharmonic 
functions and $p$-convex domains in $\r^n$ which are used in the paper,
referring to the papers of Harvey and Lawson \cite{HarveyLawson2011ALM,HarveyLawson2012AM, HarveyLawson2013IUMJ}  
and the references therein for a more complete account. 
We add the proof of Theorem \ref{th:p-convex} for unbounded domains (see Subsection \ref{ss:strongly-p-convex}) 
and formulate the Kontinuit\"atssatz for minimal submanifolds  (see Proposition \ref{prop:Kontinuitatssatz}).
In Subsection \ref{ss:MPsh-NPsh}, we recall the notion of a {\em null plurisubharmonic
function} and develop one of the main tools that will be used in the proof of 
Theorems \ref{th:main1}, \ref{th:main1-bis}, and \ref{th:main2}.

We denote by $\langle \cdotp,\cdotp\rangle$ and $\|\cdotp\|$ the standard Euclidean inner product and the 
Euclidean norm on $\r^n$, respectively.
We shall use the same notation for the Euclidean norm on $\c^n$.

%
%

\subsection{$p$-plurisubharmonic functions}
Let $\bx=(x_1,\ldots,x_n)$ be coordinates on $\r^n$.  Given a domain $D\subset\r^n$ 
and a $\Cscr^2$  function $u \colon D\to \r$, the {\em Hessian} of $u$ at a point $\bx\in D$ 
is the quadratic form $\Hess_u (\bx)=\Hess_u(\bx;\cdotp)$ on the tangent space $T_\bx\r^n\cong\r^n$, given by 
\begin{equation}\label{eq:Hess}
	\Hess_u(\bx;\xi) = \sum_{j,k=1}^n  \frac{\partial^2 u}{\partial x_j \partial x_k}(\bx)\, \xi_j \xi_k,
	\quad \xi=(\xi_1,\ldots,\xi_n) \in\r^n. 
\end{equation}
The trace of the Hessian is the Laplace operator on $\r^n$:
$
	\tr\,(\Hess_u ) = \triangle u  = \sum_{j=1}^n \frac{\di^2 u}{\di x_j^2}.
$

The Euclidean metric $ds^2=\sum_{j=1}^n dx_j\otimes dx_j$ on $\r^n$ induces 
a Riemannian metric $g=g_M$ on any smoothly immersed submanifold $M\to \r^n$. 
A function $u\in \Cscr^2(D)$ is subharmonic on a submanifold  $M\subset D$ 
if  $\triangle_M (u|_M)\ge 0$, where $\triangle_M$ is the Laplace operator on $M$ associated to the 
metric $g_M$ induced by the immersion. 
In particular, if $L$ is an affine $p$-dimensional subspace of $\r^n$ given by 
\[
	L= \bigl\{\bx(\xi) =\ba+\sum_{j=1}^p \xi_j \bv_j \in\r^n : \xi_1,\ldots,\xi_p\in\r \bigr\},
\]
where $\ba\in  \r^n$ and $\bv_1,\ldots,\bv_p\in\r^n$ is an orthonormal set, then $u$ is subharmonic on $L\cap D$ 
if and only if the function $\xi\mapsto u(\bx(\xi))$ is subharmonic on $\{\xi\in \r^p: \bx(\xi)\in D\}$.

%
%
%
%
\begin{definition}  \label{def:p-psh}
An upper semicontinuous function $u\colon D\to \r\cup\{-\infty\}$ on a domain $D \subset\r^n$ is 
{\em $p$-plurisubharmonic} for some integer $p\in\{1,\ldots,n\}$ if the restriction $u|_{L\cap D}$ 
to any affine $p$-dimensional plane $L\subset \r^n$ is subharmonic on $L\cap D$.  
A $2$-plurisubharmonic function is also called  {\em minimal plurisubharmonic}.
\end{definition}

The set of all $p$-plurisubharmonic functions on $D$ is denoted by $\Psh_p(D)$.
Following the notation introduced in \cite{DrinovecForstneric2015TAMS}, we shall write 
\[
	\Psh_2(D)=\MPsh(D).
\]
It is obvious that $\Psh_1(D) \subset \Psh_{2}(D)\subset \cdots \subset \Psh_{n}(D)$. 
An $n$-plurisubharmonic function on a domain $D\subset \r^n$ is a subharmonic function 
in the usual sense, and a $1$-plurisubharmonic function is a convex function. 
Clearly, $\Psh_p(D)$ is closed under addition and multiplication 
by nonnegative real numbers. Most of the familiar properties of plurisubharmonic functions on domains in $\c^n$ extend to 
$p$-pluri\-sub\-harmonic functions on domains in $\r^n$ (see e.g.\ \cite[Section 6]{HarveyLawson2011ALM}). 
In particular, every  $p$-plurisubharmonic function can be approximated by smooth $p$-plurisubharmonic functions.

%
%
\begin{proposition}[Proposition 2.3 and Theorem 2.13 in \cite{HarveyLawson2013IUMJ}] 
\label{prop:characterizations}  
Let $1\le p\le n$ be integers and $D$ be a domain in $\r^n$. 
The following conditions are equivalent for a function $u\in \Cscr^2(D)$:
\begin{itemize}
\item[\rm (a)]  $u$ is $p$-plurisubharmonic on $D$;
\item[\rm (b)]  $\tr_L \Hess_u(\bx )\ge 0$ for every point $\bx\in D$ and every $p$-dimensional  linear subspace $L\subset \r^n$
(here, $\tr_L$ denotes the trace of the restriction to $L$);
\item[\rm (c)]  If $\lambda_1(\bx)\le \lambda_2(\bx)\le \cdots \le \lambda_n(\bx)$ are the eigenvalues of
$\Hess_u(\bx )$, then 
\begin{equation}\label{eq:smallest-p}
	\lambda_1(\bx) + \cdots + \lambda_p(\bx)\ge 0\quad \text{for every}\ \bx\in D;
\end{equation}
\item[\rm (d)]  $u|_{M}$ is subharmonic on every minimal $p$-dimensional submanifold $M\subset D$.
\end{itemize}
\end{proposition}

\begin{proof}[Sketch of proof]
The equivalences (a)$\Leftrightarrow$(b)$\Leftrightarrow$(c) are easily seen, and  (d)$\Rightarrow$(a) is obvious. 
The nontrivial implication (b)$\Rightarrow$(d) follows from the following formula which holds for every smooth submanifold 
$M\subset \r^n$  (cf.\ \cite[Proposition 2.10]{HarveyLawson20092AJM}):
\begin{equation}\label{eq:trace-M}
	\triangle_M(u|_M) =\tr_M \Hess_u - H_M u.
\end{equation}
Here, $\tr_M \Hess_u$ is the trace of the restriction of the Hessian of $u$ to the tangent bundle of $M$ 
and $H_M$ is the mean curvature vector field  of $M$. If $M$ is a minimal submanifold, then $H_M=0$ and 
we get that $\triangle_M(u|_M) =\tr_M \Hess_u\ge 0$.
\end{proof}

%
%
\begin{definition}\label{def:sp-psh}
A function $u\in \Cscr^2(D)$ on a domain $D\subset \r^n$ is {\em strongly $p$-pluri\-subharmonic} if
$\tr_L \Hess_u(\bx )>0$ for every $p$-dimensional  affine linear subspace $L\subset \r^n$ and every point $\bx\in D\cap L$.
Equivalently, if $\lambda_1(\bx)\le \lambda_2(\bx)\le \cdots \le \lambda_n(\bx)$ are the eigenvalues of
$\Hess_u(\bx )$ then $\lambda_1(\bx) + \cdots + \lambda_p(\bx) > 0$ for all $\bx\in D$.
\end{definition}

The analogue of Proposition \ref{prop:characterizations} holds for strongly $p$-plurisubharmonic
functions; in particular, we have the following result.

%
%
\begin{proposition}\label{prop:characterizations2}
A function $u\in \Cscr^2(D)$ on a domain $D\subset \r^n$ is strongly $p$-plurisubharmonic if and only if $u|_M$ 
is strongly subharmonic on every minimal $p$-dimensional submanifold $M\subset D$.
\end{proposition}

Observe that, for any $u\in \Psh_p(D)\cap \Cscr^2(D)$, the function $u(\bx)+\epsilon \|\bx\|^2$ 
is strongly $p$-plurisubharmonic for every $\epsilon>0$. It follows that every 
$p$-plurisubharmonic function  can be approximated by smooth strongly $p$-plurisubharmonic functions.

If $h$ is a smooth real function on $\r$ and $u$ is a $\Cscr^2$ function on a domain $D\subset \r^n$, 
then for each point $\bx\in D$ and vector $\xi=(\xi_1,\ldots,\xi_n)\in\r^n$ we have
\begin{equation}\label{eq:Hess-comp}
	\Hess_{h\circ u}(\bx,\xi) = h'(u(\bx))\, \Hess_u(\bx,\xi) + h''(u(\bx))\, \|\nabla u(\bx)\cdotp \xi\|^2.
\end{equation}
It follows that, if $u$ is (strongly) $p$-plurisubharmonic and $h$ is (strongly) increasing and convex on
the range of $u$, then $h\circ u$ is also (strongly) $p$-plurisubharmonic.

%
%

\subsection{$p$-convex hulls and $p$-convex domains}
\label{ss:p-convex}

\begin{definition}[Definitions 3.1 and 3.3 in \cite{HarveyLawson2013IUMJ}] \label{def:p-hull}
Let $K$ be a compact set in a domain $D\subset \r^n$ and $p\in \{1,2,\ldots,n\}$.
The {\em $p$-convex hull} (or the {\em $p$-hull}) of $K$ in $D$ is the set
\[
	\wh K_{p,D}=\{\bx\in D : u(\bx)\le \sup_K u\ \ \text{for all}\ u\in \Psh_p(D)\}.
\]
We shall write $\wh K_{p}=\wh K_{p,\r^n}$.
The $2$-hull is also called the {\em minimal hull} and denoted
\[
	\wh K_{\Mgot,D} = \wh K_{2,D}; \qquad \wh K_{\Mgot}=\wh K_{\Mgot,\r^n}.
\]
A domain $D\subset \r^n$ is {\em $p$-convex} if  $\wh K_{p,D}$ is compact for every compact set $K\subset D$.
A $2$-convex domain is also called {\em minimally convex}.
\end{definition}

Since $\Psh_p(D)\subset \Psh_{p+1}(D)$ for $p=1,\ldots, n-1$, 
we have $K\subset \wh K_n \subset \cdots \subset 	\wh K_{2}\subset\wh K_{1} = \Co(K)$.
Simple examples show that these inclusions are strict in general. 

The following result is \cite[Theorem 3.4]{HarveyLawson2013IUMJ};
the proof is similar to the classical one concerning holomorphically convex domains in $\c^n$.

\begin{proposition}\label{prop:p-convex}
A domain $D\subset \r^n$ is {\em $p$-convex} for some $p\in \{1,2,\ldots,n\}$ 
(see Definition \ref{def:p-hull}) if and only if it admits a smooth strongly 
$p$-plurisubharmonic exhaustion function. 
\end{proposition}

The proof  of the next result follows the familiar case of plurisubharmonic functions;
see e.g.\ H\"ormander \cite[Theorem 5.1.5, p.\ 117]{Hormanderbook}.

\begin{proposition}\label{prop:p-extension}
Let $D$ be a $p$-convex domain in $\r^n$, and let $K\subset D$ be a compact $p$-convex set,
i.e., $K=\wh K_{p,D}$. Then the following conditions hold.
\begin{itemize}
\item[\rm (a)] There exists a smooth $p$-plurisubharmonic exhaustion function  $\rho\colon D\to\r_+$ 
such that $\rho^{-1}(0)=K$ and $\rho$ is strongly $p$-plurisubharmonic on $D\setminus K$.
\item[\rm (b)]
For every $p$-plurisubharmonic function $u$ on a neighborhood  $U$ of $K$ there exists
a $p$-plurisubharmonic exhaustion function $v \colon D\to\r$ which agrees with $u$ on 
$K$ and is smooth strongly $p$-plurisubharmonic on $D\setminus K$.
\end{itemize}
\end{proposition}

\begin{proof}[Proof of (a):]
For any point $\bx\in D\setminus K$ there exists a smooth strongly $p$-plurisubharmonic function $u$ on $D$ 
such that $u<0$ on $K$ and $u(\bx)>0$. Pick a smooth function $h\colon \r\to\r_+$ which equals zero on $(-\infty,0]$ and 
is strongly increasing and strongly convex on $(0,\infty)$. 
Then $h\circ u\ge 0$ vanishes on a neighborhood of $K$ and is strongly
$p$-plurisubharmonic on a neighborhood $V$ of $\bx$ in view of (\ref{eq:Hess-comp}). 
Hence  we can pick a countable collection $\{(V_j,u_j)\}_{j\in \n}$,  
where $V_j$ is an open set in $D\setminus K$, $u_j\ge 0$ is a smooth $p$-plurisubharmonic function
on $D$ that vanishes near $K$ and is strongly $p$-plurisubharmonic on $V_j$, and 
$\bigcup_{j=1}^\infty V_j= D\setminus K$. If the numbers $\epsilon_j>0$ are chosen small enough, then
the series $v=\sum_{j=1}^\infty \epsilon_j u_j \ge 0$ converges in $\Cscr^\infty(D)$. By the construction, $v$
vanishes precisely on $K$ and is strongly $p$-plurisubharmonic on $D\setminus K$. 
Finally, take $\rho = v+h\circ \tau$, where $\tau$ is a smooth $p$-plurisubharmonic 
exhaustion function on $D$ that is negative on $K$.

\noindent {\em Proof of (b):}
We may assume that $\overline U$ is compact. Choose a smooth function $\chi$ on $\r^n$ 
such that $\chi=1$ on a neighborhood of $K$ and $\supp \,\chi\subset U$. 
Let $\rho$ be as in part (a). The function $v=\chi u +C\rho$ then satisfies condition (b) 
if the constant $C>0$ is chosen big enough. Indeed, the (very) positive Hessian
of $C\rho$ compensates the bounded negative part of the Hessian of $\chi u$ on the compact 
support of $d\chi$ which is contained  in $U\setminus K$.
\end{proof}

%
%

\subsection{Domains with smooth $p$-convex boundaries}
\label{ss:strongly-p-convex}

%
%
%
%
\begin{proof}[Proof of Theorem \ref{th:p-convex}]
As pointed out in the Introduction, these results were proved by Harvey and Lawson 
\cite{HarveyLawson2013IUMJ} for bounded domains; here we extend their arguments to unbounded domains. 

Thus, let $D\subset \r^n$ be a domain with boundary $bD$ of class $\Cscr^2$. Assume first that condition (a) holds,
i.e., $D$ is $p$-convex. It is immediate that such $D$ is also locally $p$-convex, in the sense that every point 
$\bx \in bD$ has a neighborhood $U\subset \r^n$ such that $D\cap U$ is $p$-convex 
(cf.\ \cite[(3.1) and Theorem 3.7]{HarveyLawson2013IUMJ}; the cited results also give the converse implication 
for bounded domains).  Furthermore, local $p$-convexity admits the following differential theoretic 
characterization  (cf.\ \cite[Remark 3.11]{HarveyLawson2013IUMJ}):

{\em A smoothly bounded domain $D\subset \r^n$ is locally $p$-convex at $\bx\in bD$ if and only if 
there are a neighborhood $U\subset \r^n$ of $\bx$ and a local smooth defining function $\rho$ for $D$ 
(i.e., $D\cap U=\{\rho<0\}$ and $d\rho\ne 0$ on $bD\cap U=\{\rho=0\}$) such that
\[
	\tr_L \Hess_\rho(\by) \ge 0\quad \text{for every tangent $p$-plane}\ L\subset T_\by bD,\
	\by\in bD\cap U.
\] 
}
This property is independent of the choice of $\rho$
and is equivalent to property (c) in Theorem \ref{th:p-convex} (that the sum of $p$ smallest principal curvatures of $bD$ is
nonnegative). Furthermore, setting $\delta=\dist(\cdotp,bD)$, $D$ is locally $p$-convex if and only if the function $-\log\delta$ is
$p$-plurisubharmonic on a collar around $bD$ in $D$ (cf.\ \cite[Summary 3.16]{HarveyLawson2013IUMJ}). 

This justifies the implications (a)$\Rightarrow$(b)$\Leftrightarrow$(c)$\Rightarrow$(d)  in Theorem \ref{th:p-convex}.

It remains to prove that (d)$\Rightarrow$(a). Assume that (d) holds, i.e., the $\Cscr^2$ function 
$-\log\delta$ is $p$-plurisubharmonic on an interior collar $U\subset D$ around $bD$.  
Choose a smooth cut-off function $\chi\colon \r^n\to [0,1]$
which equals $0$ on an open set  $V\subset D$ containing $D\setminus U$ and equals $1$ on an open set
$W\subset \r^n$ containing $\r^n\setminus D$. Its differential $d\chi$ has support in the set $U\setminus W$
whose closure is contained in $D$. The product $-\chi \log \delta$ is then a function of class $\Cscr^2(D)$
which is $p$-plurisubharmonic near $bD$ and tends to $+\infty$ along $bD$.
Let $h\colon \r_+\to\r_+$ be a smooth, increasing, strongly convex function. If $h$ is chosen such that its
derivative $h'(t)>0$ grows sufficiently fast as $t\to+\infty$, then we see from (\ref{eq:Hess-comp}) 
that the function
\[
	\rho(\bx) = -\chi(\bx) \log \delta(\bx) + h(\|\bx\|^2), \quad \bx\in D
\]
is a strongly $p$-plurisubharmonic exhaustion function on $D$, so condition (a) holds.
\end{proof}

\begin{corollary}
A domain  $D\subset \r^n$ (not necessarily bounded), 
whose boundary $bD$ is a smooth embedded minimal hypersurface,
is $(n-1)$-convex (also called mean-convex, see Remark \ref{rem:mean-convex}). 
In particular, a domain in $\r^3$ bounded  by a closed embedded minimal surface is minimally convex.
\end{corollary}

%
%

\subsection{The Maximum Principle and the Kontinuit\"atssatz} 
\label{ss:maximum}
Since the restriction of a $p$-plurisubharmonic function $u$ on a domain $D\subset \r^n$ to a minimal $p$-dimensional
submanifold $M\subset D$ is subharmonic on $M$ (cf.\ Propositions \ref{prop:characterizations}
and \ref{prop:characterizations2}), it follows from the maximum principle for subharmonic functions that, 
for any compact minimal $p$-dimensional submanifold $M\subset D$ with boundary $bM$, we have the  implication
\[
 	bM\subset K \Longrightarrow M\subset \wh K_{p,D}.
\] 
The same conclusion holds for immersed minimal submanifolds and for minimal $p$-dimensional currents.
Furthermore, we have the following result which is analogous to the classical 
{\em Kontinuit\"atssatz} (also called the {\em continuity principle}) in complex analysis.
(Compare with Harvey and Lawson \cite{HarveyLawson2013IUMJ}, proof of Theorem 3.9 on p.\ 159.)

%
%

\begin{proposition}[Kontinuit\"atssatz for minimal submanifolds] 
\label{prop:Kontinuitatssatz}
Assume that $D$ is a $p$-convex domain in $\r^n$ for some $p\in \{1,\ldots,n\}$ and 
$\{M_t\}_{t\in [0,1)}$ is a continuous family of immersed compact minimal $p$-dimensional submanifolds 
of $\r^n$ with boundaries $bM_t$. If $M_0\subset D$ and 
$\bigcup_{t\in [0,1)} bM_t$ is contained in a compact subset of $D$, then $\bigcup_{t\in [0,1)} M_t$ is also
contained in a compact subset of $D$. 
\end{proposition}

\begin{proof}
Let $K$ denote the closure  of the set $M_0\cup  \bigcup_{t\in [0,1)} bM_t$ in $D$.
By the hypothesis, $K$ is compact. Since $D$ is $p$-convex, the $p$-hull $L=\wh K_{p,D}\subset D$ of $K$ 
is also compact.   Consider the set $J=\{t\in [0,1): M_t\subset L\}$. We have $0\in J$ by the hypothesis.
We claim that $J=[0,1)$. Since the family $M_t$ is continuous in $t$ and $L$ is compact, 
$J$ is closed. It remains to see that $J$ is also open.  Assume that $t_0\in J$; then $M_{t_0}\subset L\subset D$.
By continuity, it follows that $M_t\subset D$ for all $t\in [0,1)$ sufficiently close to $t_0$, 
and the maximum principle implies that $M_t\subset L$ for all such $t$.
\end{proof}

\begin{problem}
Assume that $1<p<n$ and $D$ is a domain in $\r^n$  which satisfies the conclusion
of Proposition \ref{prop:Kontinuitatssatz} for minimal $p$-dimensional 
submanifolds. Does it follow that $D$ is $p$-convex? 
Is the function $-\log\dist(\cdotp,bD)$  $p$-plurisubharmonic  on $D$?
\end{problem}

If $bD$ is smooth, then the validity of the Kontinuit\"atssatz for $D$ implies (by Harvey and Lawson,
cf.\ Theorem  \ref{th:p-convex} above) that $-\log\dist(\cdotp,bD)$ is $p$-plurisubharmonic
near $bD$; even in this case, it is not clear whether it is $p$-plurisubharmonic on all of $D$.
The analogous result in complex analysis is Oka's theorem, saying 
that the function $-\log\dist(\cdotp,bD)$ is plurisubharmonic on any Hartogs pseudoconvex 
domain $D\subset\c^n$ (see e.g.\ \cite[Theorem 5.6, p.\ 96]{Rangebook}).
Its proof breaks down in the present situation since the sum of two minimal discs in $\r^n$ is not 
a minimal disc in general.

The following result will be used in the proof  of Theorem \ref{th:FTC} in Section \ref{sec:maximal}.

%
%
%
%

\begin{proposition}[The Maximum Principle for minimal submanifolds] 
\label{prop:max-principle}
Let $D$ be a proper $p$-convex domain in $\r^n$ and let $M\subset D$ be a compact, connected, immersed 
minimal $p$-dimensional submanifold with boundary $bM$. Then the following hold:
\begin{itemize}
\item[\rm (a)]  $\dist(bM,bD) = \dist(M,bD)$.   
\item[\rm (b)]  If $D$ has smooth boundary and $\dist(\bx_0,bD) = \dist(bM,bD)$
for some point $\bx_0 \in \mathring M=M\setminus bM$, then $bD$ contains a translate of $M$.
\item[\rm (c)] If the assumption in part (b) holds for $p=2$, $n=3$ 
(i.e., $M$ is a compact minimal surface with boundary in a minimally convex domain $D\subset \r^3$
and $\dist(\bx_0,bD) = \dist(bM,bD)$  for some $\bx_0 \in \mathring M$), then $M$ is a piece of a plane. 
Moreover, if $\by_0\in bD$ is such that $\|\bx_0-\by_0\| = \dist(bM,bD)$, then 
$\bigcup_{t\in [0,1)} t(\by_0-\bx_0) +M\subset D$ and $(\by_0-\bx_0)+M\subset   bD$.
\end{itemize}
\end{proposition}

\begin{proof}[Proof of (a)]
Assume that  $\dist(\bx_0,bD) < \dist(bM,bD)$ for some $\bx_0 \in \mathring M$. Pick a point $\by_0 \in bD$ 
such that $\dist(\bx_0,bD)=\|\bx_0-\by_0\|$ and a number $t_0$ with $\|\bx_0-\by_0\|  < t_0 < \dist(bM,bD)$.
The family of translates $M_t=M+t(\by_0-\bx_0)/\|\by_0-\bx_0\|$ for $t\in [0,t_0]$ then violates 
Proposition \ref{prop:Kontinuitatssatz}. This contradiction proves part (a).

\noindent {\em Proof of (b).} 
By Theorem \ref{th:p-convex}, there are a neighborhood $U\subset \r^n$ of $bD$ 
and a $p$-plurisubharmonic function $\rho$ on $U$ such that $U\cap  D=\{\bx\in U\colon \rho(\bx)<0\}$.
Let $\bx_0\in\mathring M$ be such that $c=\dist(\bx_0,bD)= \dist (M,bD)$. 
Pick a point $\by_0\in bD$ with $\|\bx_0-\by_0\|=c$.  
There is a compact connected neighborhood $V \subset M$ of 
$\bx_0$ in $M$ such that the translate $W= V + \by_0-\bx_0$  
is contained in $U$, and hence in $U\cap \overline D$ by part (a).  
Clearly $\by_0\in W$. Since the function $\rho|_W\le 0$ 
is subharmonic and $\rho(\by_0)=0$, it is constantly equal to zero by the maximum principle, 
and hence $W\subset bD$. This means that, for every $\bx\in V$, we have 
\begin{equation}\label{eq:dist}
	\bx+\by_0-\bx_0\in bD\quad  \text{and}\quad  \dist(\bx,bD)=\dist(M,bD).
\end{equation}
This argument shows that the set of points $\bx\in M$ satisfying (\ref{eq:dist}) 
is open, and clearly it is also closed, so it equals $M$.  Thus $M+\by_0-\bx_0 \subset bD$ .

\noindent {\em Proof of (c).}  Let $\bx_0$ and $\by_0$ be as in the statement of (c). Then $M$ does not intersect
the open ball centered at $\by_0$ of radius $\|\bx_0-\by_0\| = \dist(bM,bD)$.
This implies that $\by_0 =\bx_0+c {\rm N}(\bx_0)$,
where ${\rm N}\colon V\to \s^2$ is a Gauss map of the orientable surface $V\subset M$ introduced in part (b).
Since (\ref{eq:dist})  holds for all $\bx\in V$, we see that  ${\rm N}(\bx)={\rm N}(\bx_0)$ for all $\bx\in V$.
This shows that  $V$, and hence also $M$, is a piece of a plane and (c) follows.
\end{proof}

%
%
%
%
\subsection{Null plurisubharmonic functions}
\label{ss:MPsh-NPsh}
Let $\bz=(z_1,\ldots,z_n)=\bx+\imath \by$ be complex coordinates on $\c^n$, with $z_j=x_j+\imath y_j$ for 
$j=1,\ldots,n$. We shall write $\zero=(0,\ldots,0)$ for the origin in $\r^n$ or in $\c^n$.
Given a $\Cscr^2$ function $\rho \colon  \Omega \to\r$ on a domain $\Omega \subset \c^n$,
we denote by $\Lcal_\rho(\bz;\cdotp)$ its {\em Levi form} at a point $\bz\in \Omega$:
\begin{equation}\label{eq:Levi}
	\Lcal_\rho(\bz;\bw) = \sum_{j,k=1}^n \frac{\partial^2 \rho}{\partial z_j \partial\bar z_k}(\bz) w_j\overline w_k,
	\quad \bw=(w_1,\ldots,w_n)\in\c^n.
\end{equation}

We shall use the following lemma whose proof amounts to a simple calculation.

\begin{lemma}\label{lem:relation}
Let  $B=(b_{j,k})$ be a real symmetric $n\times n$ matrix, and let $\bw=\bu+\imath \bv\in\c^n$. Then 
\[
	2\sum_{j,k=1}^n b_{j,k} u_j u_k = 
	\Re\left( \sum_{j,k=1}^n b_{j,k} w_j w_k  \right) + \sum_{j,k=1}^n b_{j,k} w_j\overline w_k.
\]
\end{lemma}

A function $\rho\colon D\to \r$ on a domain $D\subset\r^n$ will also be considered as
a function on the tube $\Tcal_D= D\times\imath \r^n \subset\c^n$ 
which is independent of the imaginary variable:
\begin{equation}\label{eq:independent}
	\rho(\bx+\imath \by)=\rho(\bx)\quad \text{for all}\ \bx\in D\ \text{and}\ \by\in\r^n.
\end{equation}
Fix a point $\bx\in D$ and a vector $\bu\in \r^n$. The Hessian $\Hess_\rho(\bx;\cdotp)$  
 (\ref{eq:Hess}) has coefficients
\[
	b_{j,k} := \frac{\partial^2 \rho}{\partial x_j \partial x_k}(\bx) =
	4 \frac{\partial^2 \rho}{\partial z_j \partial\bar z_k}(\bx) \in\r.
\] 
Lemma \ref{lem:relation} shows that, for every  $\bw=\bu+\imath\bv \in \c^n$, we have
\begin{equation}\label{eq:Hess-Levi}
	 \frac{1}{2}\, \Hess_\rho(\bx;\bu) = \frac{1}{4}\,  \Re\left( \sum_{j,k=1}^n b_{j,k} w_j w_k  \right)
	+ \Lcal_\rho(\bx;\bw).
\end{equation}
Replacing $\bw$ by $-\imath \bw=\bv-\imath \bu$ and noting that 
$\Lcal_\rho(\bx;\pm \imath \bw)=\Lcal_\rho(\bx;\bw)$ while the first term on the right hand side
of  (\ref{eq:Hess-Levi}) changes sign, we obtain 
\[
	 \Hess_\rho(\bx;\bu) +  \Hess_\rho(\bx;\bv) = 4 \Lcal_\rho(\bx;\bu+\imath \bv).
\]
In particular, if $(\bu,\bv)$ is an orthonormal pair of vectors in $\r^n$ and we set
\[ 
	L= \bx + \span_\r\{\bu,\bv\}\subset \r^n,\quad 
	\Lambda=\bx +  \span_\c\{\bu+\imath \bv\}\subset \c^n,
\] 
then it follows that
\begin{equation}\label{eq:triangle}
	\triangle (\rho|_L)(\bx) = 4 \Lcal_\rho(\bx;\bu+\imath \bv) = \triangle (\rho|_\Lambda)(\bx).
\end{equation}

Set $a_j=\frac{\di \rho}{\di x_j}(\bx)\in\r$ for $j=1,\ldots,n$. The identity (\ref{eq:Hess-Levi}) implies
that, for every point $\bz=\bx+\imath\by\in \Tcal_D$  and vector $\bw=\bu+\imath\bv\in \c^n$ near $\zero\in\c^n$,
we have the Taylor expansion
\begin{eqnarray*}  \label{eqn:Taylor}
	\rho(\bz+\bw) &=& \rho(\bx) + \sum_{j=1}^n a_j u_j + \frac{1}{2}\, \Hess_\rho(\bx;\bu) + o(\|\bu\|^2) \\
	&=&  \rho(\bx) +\Re\left( \sum_{j=1}^n a_j w_j  +   \frac{1}{4} \sum_{j,k=1}^n b_{j,k} w_j w_k   \right) 
	+ \Lcal_\rho(\bx;\bw) + o(\|\bw\|^2).   
\end{eqnarray*}
Denote by $\Sigma_{\bx}\subset \c^n$ the local complex hypersurface near the origin $\zero\in \c^n$ given by 
\begin{equation}\label{eq:Sigma-z}
	\Sigma_{\bx}= \bigl\{\bw : \sum_{j=1}^n a_j w_j  +   \frac{1}{4} \sum_{j,k=1}^n b_{j,k} w_j w_k=0 \bigr\}.
\end{equation}
It follows that 
\begin{equation}\label{eq:restriction}
	\rho(\bz+\bw) = \rho(\bz) +  \Lcal_\rho(\bx;\bw) + o(\|\bw\|^2),\qquad  
		\bz=\bx+\imath\by \in \Tcal_D,\ \bw\in \Sigma_\bx.
\end{equation}

We need to recall the connection between minimal plurisubharmonic functions on a domain $D\subset\r^n$ 
and null plurisubharmonic functions on the tube $\Tcal_D=D\times \imath \r^n\subset \c^n$;
the latter class of functions was introduced in \cite{DrinovecForstneric2015TAMS}. 
%
%
%
%

Let $\Agot\subset \c^n$ denote the {\em null quadric}:
\begin{equation}\label{eq:Agot}
	\Agot =\{\bz=(z_1,\ldots,z_n) \in\c^n : z_1^2+z_2^2+\cdots + z_n^2=0\},
	\quad \Agot_*=\Agot\setminus\{\zero\}.
\end{equation}

%
%
\begin{definition}[Definitions 2.1 and 2.4 in \cite{DrinovecForstneric2015TAMS}]  \label{def:npsh}
Let $\Omega$ be a domain in $\c^n$ for some $n\ge 3$. 
\begin{itemize}
\item[\rm (a)]  An upper semicontinuous function $u\colon \Omega \to\r\cup\{-\infty\}$ is
{\em null plurisubharmonic} ($u\in \NPsh(\Omega)$) if, for any affine complex line 
$L\subset \c^n$ directed by a null vector $\theta\in \Agot_*$, the restriction of $u$ to $L\cap \Omega$ is subharmonic. 
(If $u\in\Cscr^2(\Omega)$, this is equivalent to the condition that
$\Lcal_u(\bz;\bw)\ge 0$ for every $\bz\in \Omega$ and $\bw \in\Agot_*$.)

\vspace{1mm}
\item[\rm (b)]
A function $u\in \Cscr^2(\Omega)$ is {\em null strongly plurisubharmonic} 
if $\Lcal_u(\bz;\bw)>0$ for every $\bz\in \Omega$ and $\bw \in\Agot_*$. 
\end{itemize}
\end{definition}

Note that a vector $0\ne \bw=\bu+\imath\bv \in \c^n$ belongs to the null quadric $\Agot$
(\ref{eq:Agot}) if and only if the vectors $\bu,\bv\in \r^n$ are orthogonal and have equal length:
\begin{equation}\label{eq:conformal-null} 
	\bu+\imath\bv \in \Agot_* \ \Longleftrightarrow\     \bu\cdotp \bv =0\ \text{and}\ \|\bu\|=\|\bv\|.
\end{equation}

Assume that $(\bu,\bv)$ is an orthonormal pair in $\r^n$.
In view of  (\ref{eq:triangle}), we have the following result for functions  $u \in\Cscr^2(D)$; the general case 
for upper semicontinuous functions is seen similarly. 

%
%
%
%
\begin{lemma}[Lemma 4.3 in \cite{DrinovecForstneric2015TAMS}]  \label{lem:minimal}
Let $D$ be a domain in $\r^n$ and $\Tcal_D =D\times \imath \r^n\subset\c^n$.
\begin{itemize}
\item If $u$ is (strongly) minimal plurisubharmonic on $D$, then the function 
$U(\bx+\imath \by)=u(\bx)$ is (strongly) null plurisubharmonic on $\Tcal_D$. 
\vspace{1mm}
\item Conversely, assume that a function $U\colon \Tcal_D \to\r$ is independent of  the variable $\by=\Im \bz$,
and let $u(\bx)=U(\bx+\imath \zero)$ for $\bx\in D$. 
If $U$ is (strongly) null plurisubharmonic on $\Tcal_D$, then $u$ is (strongly) minimal plurisubharmonic on $D$. 
\end{itemize}
\end{lemma}

Recall that a {\em null holomorphic disc} in $\c^n$ $(n\ge 3)$   is a holomorphic map 
$F=(F_1,\ldots,F_n) \colon \d\to\c^n$ satisfying the nullity condition $F'(\zeta)\in \Agot$;
equivalently:
\begin{equation} \label{eq:Ndisc}
		F'_1(\zeta)^2 + F'_2(\zeta)^2  +\cdots +  F'_n(\zeta)^2=0,
		\quad \zeta \in \d.
\end{equation} 
More generally, a holomorphic immersion $F\colon M\to\c^n$ from an open Riemann surface $M$ is a 
{\em holomorphic null curve} if the derivative of $F$ in any local holomorphic coordinate on $M$
satisfies the condition (\ref{eq:Ndisc}). It follows from (\ref{eq:conformal-null}) and the
Cauchy-Riemann equations that the real and the imaginary part of a holomorphic null disc 
$F\colon \d\to\c^n$ are conformal minimal discs in $\r^n$; conversely, every conformal minimal
disc is the real part of a holomorphic null disc. We have the following observation.

\begin{proposition}[Proposition 2.7 in \cite{DrinovecForstneric2015TAMS}] \label{prop:testing-null}
An upper semicontinuous function $u$ on a domain $\Omega\subset \c^n$ $(n\ge 3)$ is null plurisubharmonic
if and only if the function $u\circ F$ is subharmonic on $\d$ for every null holomorphic disc $F\colon\d\to\Omega$.
\end{proposition}

%
%
%
%
\section{Proof of Theorems \ref{th:main1},  \ref{th:main1-bis}, and \ref{th:main2}}
\label{sec:Proof1}
 
We begin with technical preparations.
 
Let $\rho\colon D\to \r$ be a smooth minimal strongly plurisubharmonic exhaustion function on a domain 
$D\subset \r^3$. We extend $\rho$ to a function on the tube $\Tcal_D = D\times \imath\r^3\subset\c^3$
which is independent of the imaginary variable; see (\ref{eq:independent}). 
By Lemma \ref{lem:minimal}, the extended function $\rho$ is null strongly plurisubharmonic on $\Tcal_D$. 
For every point $\bx\in  D$ we denote by $\Sigma_\bx\subset \c^3$ the local complex 
hypersurface at $\zero\in\c^3$ given by (\ref{eq:Sigma-z}): 
\begin{equation}\label{eq:Sigma-z2}
	\Sigma_{\bx}= \bigl\{\bw=(w_1,w_2,w_3) : \sum_{j=1}^3 \frac{\di \rho}{\di x_j}(\bx)  w_j  
	+   \sum_{j,k=1}^3  \frac{\partial^2 \rho}{\partial z_j \partial\bar z_k}(\bx) w_j w_k=0 \bigr\}.
\end{equation}
Let $P$ denote the critical locus of $\rho$. We assume in the sequel that $\bx \in D\setminus P$;
then $\Sigma_\bx$ is nonsingular at $\zero\in \Sigma_\bx$ and its tangent space is
\begin{equation}\label{eq:tangent}
	T_\zero \,\Sigma_\bx = \bigl\{\bw \in\c^3 : \sum_{j=1}^3 \frac{\di \rho}{\di x_j}(\bx)   w_j  =0\bigr\}.
\end{equation}
Note that the coefficients $a_j=\frac{\di \rho}{\di x_j}(\bx)$ of the equation in (\ref{eq:tangent}) are real. 
By shrinking $\Sigma_\bx$ around $\zero$ if necessary, we may assume that the hypersurface $\Sigma_\bx$ is nonsingular.

The intersection of the null quadric $\Agot$ (\ref{eq:Agot})  with any complex $2$-plane 
$\zero\in \Lambda\subset \c^3$ consists of two complex lines which may coincide for certain $\Lambda$. 
 However, for a $2$-plane $\Lambda=\bigl\{\bw=(w_1,w_2,w_3)\in\c^3\colon \sum_{j=1}^3 a_j w_j =0 \bigr\}$
with real coefficients $a_1,a_2,a_3\in\r$ not all equal to 0,
the intersection $\Agot\cap \Lambda$ consists of two distinct 
complex lines as is seen by a simple calculation.
Identifying the tangent space $T_\bz\c^3$ with $\c^3$, we may consider the null quadric $\Agot$
as a subset of $T_\bz\c^3$ for any point $\bz\in\c^3$. By what has been said above,
for any point $\bz\in \Sigma_\bx$ sufficiently close to $\zero$ the intersection $\Agot\cap T_\bz \Sigma_\bx$ 
is a union of two distinct complex lines. This defines on $\Sigma_\bx$ a couple of holomorphic direction fields, 
and hence (by integration) a couple of one dimensional complex analytic foliations by 
holomorphic null curves. In particular, for any point $\bx\in D\setminus P$ we have 
two distinct embedded holomorphic null discs $\Ncal^1_\bx, \Ncal^2_\bx \subset \Sigma_\bx$ passing through $\zero$. 
Although there is no well defined global ordering of these two null discs when $\bx$ runs over $D\setminus P$, 
such an ordering clearly exists on every simply connected subset. 
By the definition of $\Sigma_\bx$ and  (\ref{eq:restriction}), we have that
\[ 
	\rho(\bz+\bw) = \rho(\bz) +  \Lcal_\rho(\bx;\bw) + o(\|\bw\|^2),\quad  
		\bw\in \Sigma_\bx.
\]  
This holds in particular for all $\bw\in \Ncal^1_\bx \cup \Ncal^2_\bx\subset \Sigma_\bx$. 
Since $\rho$ is null strongly plurisubharmonic on $\Tcal_D$, the Levi form $\Lcal_\rho(\bx;\cdotp)$
is positive on the null lines $T_\zero  \Ncal^j_\bx$ for $j=1,2$. It follows that  
for every point $\bz=\bx+\imath\by \in \Tcal_D$ with $\bx\in D\setminus P$ 
there exist constants $C_\bx>0$ and $\delta_\bx>0$ such that
\begin{equation}\label{eq:restriction2}
	\rho(\bz+\bw) \ge \rho(\bz) + C_\bx \|\bw\|^2,\quad  \bw\in \Ncal^1_\bx \cup \Ncal^2_\bx,\ \|\bw\|\le \delta_\bx.
\end{equation}
Moreover, the constants $C_\bx$ and $\delta_\bx$ can clearly be chosen uniform
for all points $\bx$ in any given compact subset of $D\setminus P$.
By projecting the discs $\Ncal^1_\bx,\ \Ncal^2_\bx$ to $\r^3$ we get a corresponding family of 
conformal minimal discs with the analogous properties.

We summarize the above discussion in the following lemma.

%
%
\begin{lemma}\label{lem:M-discs}
Let $D$ be a domain in $\r^3$, and let $\rho\colon D \to \r$ be a $\Cscr^2$ minimal strongly plurisubharmonic 
function with the critical locus $P$. For every compact set $L \subset D\setminus P$ there exist a constant 
$c=c_L>0$ and families of embedded null holomorphic discs
$\sigma_\bx^j = \alpha_\bx^j + \imath  \beta_\bx^j \colon \cd\to \c^3$  $(\bx\in L,\ j=1,2)$,
depending locally $\Cscr^1$ smoothly on the point $\bx\in L$ and satisfying the following conditions:
\begin{itemize}
\item[\rm (a)]   $\sigma_\bx^j(0)=0$;
\item[\rm (b)]   $\{\bx + \alpha_\bx^j(\zeta) : \zeta\in \cd\} \subset D$;
\item[\rm (c)]   the function $\cd\ni \zeta\mapsto \rho\bigl(\bx+ \alpha_\bx^j(\zeta)\bigr)$ is strongly convex 
and satisfies
\begin{equation}\label{eq:estimate-c}
	\rho\bigl(\bx + \alpha_\bx^j(\zeta)\bigr) \ge \rho(\bx)+c \|\zeta\|^2,\quad \zeta\in \cd.   
\end{equation}
\end{itemize}
\end{lemma}

The conformal minimal discs $\alpha_\bx^j\colon \cd\to\r^3$, furnished by Lemma \ref{lem:M-discs},
will be used to push the boundary $F(bM)$ of a given conformal minimal immersion $F\colon M\to  D$ 
to a higher level set of $\rho$, except near the critical points of $\rho$ which shall be avoided
by a different method explained in the sequel. The relevant tool for this lifting  is the following. 
(Related results on the Riemann-Hibert problem for null curves are given by \cite[Theorem 4]{AlarconForstneric2015MA}
in dimension $n=3$, and by \cite[Theorem 3.5]{AlarconDrinovecForstnericLopez2015} in arbitrary dimension $n\ge 3$.)

%
%
%
%
\begin{theorem}[Riemann-Hilbert problem for conformal  minimal surfaces in $\r^3$]
\label{th:RH}
Let $M$ be a compact bordered Riemann surface with nonempty boundary $bM\ne\emptyset$, 
let $I_1,\ldots,I_k$ be pairwise disjoint compact subarcs of $bM$ which are not 
connected components of $bM$, and set $I=\bigcup_{j=1}^k I_j$.  
Choose a thin annular neighborhood  $A\subset M$ of $bM$ 
and a smooth retraction $\rho\colon A\to bM$.  Assume that
\begin{itemize}
\item $F\colon M \to\r^3$ is a conformal minimal immersion of class $\Cscr^1(M)$,
\item $r \colon bM \to [0,1]$ is a continuous function supported on  $I$, and
\item $\alpha \colon I \times\overline{\d}\to\r^3$ is a map of class $\Cscr^1$ such that
for every $\zeta\in I$ the map $\cd \ni \xi \mapsto \alpha(\zeta,\xi)\in\r^3$ is a conformal minimal 
immersion with $\alpha(\zeta,0)=0$.
\end{itemize}
Let  the map $\varkappa\colon bM \times \overline{\d}\to \r^3$ be given by
\begin{equation}\label{eq:varkappa}
	\varkappa(\zeta,\xi)=F(\zeta) + \alpha\bigl(\zeta,r(\zeta) \,\xi\bigr),
\end{equation}
where we take $\alpha\bigl(\zeta,r(\zeta) \,\xi\bigr)=0$ for $\zeta\in bM\setminus I$.
Given a number $\eta>0$ and an open neighborhood $\Omega\subset M$ of   
$I$, there exists a conformal minimal immersion $G\colon M\to\r^3$ of class $\Cscr^1(M)$ 
satisfying the following conditions:
\begin{enumerate}[\it i)]
\item $\dist(G(\zeta ),\varkappa(\zeta ,\t))<\eta$ for all $\zeta \in bM$;
\item $\dist(G(\zeta ),\varkappa(\rho(\zeta ),\overline{\d}))<\eta$ for all $\zeta \in \Omega$;
\item $\|G-F\|_{1,M\setminus \Omega}<\eta$;
\item $\Flux(G)=\Flux(F)$.									                                 
\end{enumerate}
\end{theorem}

\begin{proof}
If $M$ is the disc $\cd$, the conclusion  follows from \cite[Lemma 3.1]{AlarconDrinovecForstnericLopez2015}
which gives an analogous result for  null holomorphic immersions in $\c^3$. Since
every conformal minimal disc $\cd\to \r^3$ is the real part of a holomorphic null disc $\cd\to\c^3$,
the cited lemma can be used for the corresponding families of null discs;  the real 
part $G$ of the resulting null  disc  then satisfies the conclusion of Theorem \ref{th:RH}. 
(The loss of smoothness in harmonic conjugates is not important 
since we can restrict  our maps to a slightly smaller disc.) 

In the general case, for an arbitrary bordered Riemann surface $M$, 
one follows the proof of \cite[Theorems 3.5 and 3.6]{AlarconDrinovecForstnericLopez2015}, but replacing 
\cite[Lemma 3.3]{AlarconDrinovecForstnericLopez2015} by  \cite[Lemma 3.1]{AlarconDrinovecForstnericLopez2015}. 
The former one holds in any dimension $n\ge 3$,  but only applies to 
flat conformal minimal discs $\alpha(\zeta,\cdotp)\colon\cd\to\r^n$ lying in parallel
$2$-planes, while  the latter one holds without any such restriction on $\alpha$, but only in dimension $n=3$.
\end{proof}

The next result  is the main technical ingredient in the proofs of Theorems \ref{th:main1},  
\ref{th:main1-bis}, and \ref{th:main2}.  Similar techniques have been used for lifting boundaries 
of complex curves and Stein varieties in $q$-convex manifolds;
see e.g.\ \cite{DrinovecForstneric2007DMJ,DrinovecForstneric2010AJM} and the references therein.

%
%
%
%
\begin{proposition}[Lifting boundaries of conformal minimal surfaces]
\label{prop:lifting}
Let $D$ be a domain in $\r^3$ and $\rho\colon D \to \r$ be a $\Cscr^2$ minimal strongly plurisubharmonic 
function with the critical locus $P$. Given a compact set $L \subset D\setminus P$, there exist constants
$\epsilon_0>0$ and $C_0>0$ such that the following holds.

Let $M$ be a compact bordered Riemann surface, and let $F\colon M\to D$ be a conformal minimal immersion 
of class $\Cscr^1(M)$. Given a continuous function $\epsilon \colon bM\to [0,\epsilon_0]$ supported on the set
$J=\{\zeta\in bM: F(\zeta) \in L\}$, an open set $U\subset M$ containing $\supp(\epsilon)$ in its relative interior, 
and a constant $\delta>0$, there exists a conformal minimal immersion 
$G\colon M\to D$ satisfying the following conditions:
\begin{enumerate}
\item $|\rho(G(\zeta)) - \rho(F(\zeta)) -\epsilon(\zeta)| <\delta$ for every $\zeta\in bM$;           
\item $\rho(G(\zeta))\ge \rho(F(\zeta)) -\delta$ for every $\zeta\in M$;                                      
\item $\|G-F\|_{1,M\setminus U}<\delta$;      				                                          
\item $\|G-F\|_{0,M} \le C_0 \sqrt{\epsilon_0}$; 								      
\item $\Flux(G)=\Flux(F)$.									                                 
\end{enumerate} 
\end{proposition}

\begin{proof}
By approximation, we may assume that $F$ is of class $\Cscr^\infty(M)$ 
(see \cite{AlarconForstnericLopez2016MZ,AlarconLopez2012JDG}).

Pick a compact set $L_0\subset D\setminus P$ which contains $L$ in its interior. 
Let $c_{L_0}$ be the constant  furnished by Lemma \ref{lem:M-discs} for the set $L_0$,
and choose a number $\epsilon_0$ such that $0<\epsilon_0< c_{L_0}$.
Set $J_0=\{\zeta\in bM: F(\zeta) \in L_0\}$.  By approximation, we may assume that 
the function $\epsilon \colon bM\to [0,\epsilon_0]$ in Proposition \ref{prop:lifting} is smooth
and supported in the relative interior of $J_0\cap U$. 

Assume first that the support of $\epsilon$ does not contain any boundary curves of $M$; 
the general case will be obtained by two consecutive  applications of this special case.
Choose finitely many closed pairwise disjoint segments  $I_1,I_2,\ldots, I_m\subset  J_0 \cap U$ 
whose union $I=\bigcup_{j=1}^m I_j$ contains $\supp(\epsilon)$ in its relative interior.
Note that $F(I)\subset L_0$. Since $I$ is simply connected, Lemma \ref{lem:M-discs} 
(see in particular (\ref{eq:estimate-c})) furnishes 
a family of conformal minimal discs $\alpha_{F(\zeta)}\colon\cd\to \r^3$, depending smoothly
on  $\zeta\in I$, such that 
\begin{equation}\label{eq:lower-estimate}
	\rho(F(\zeta)+\alpha_{F(\zeta)}(\xi))\ge \rho(F(\zeta)) + c_{L_0}> 
	\rho(F(\zeta)) + \epsilon_0, \quad   \zeta\in I,\   |\xi|=1.
\end{equation}
Without loss of generality we may assume that $\delta<3 \epsilon_0$. 
Let $\tilde \epsilon\colon I\to [\delta/3,\epsilon_0]$ be obtained by smoothing the function
$\max\{\epsilon,\delta/3\}$; in particular, we assume that $\tilde\epsilon=\epsilon$ on the set where 
$\epsilon\ge \delta/2$ and $\delta/3 \le \tilde\epsilon < \delta/2$ on the complementary set. 
The properties of the discs $\alpha_{\bx}$, furnished by Lemma \ref{lem:M-discs}, imply 
that for every fixed $\zeta\in I$ the function $\d\ni \xi \mapsto \rho(F(\zeta) + \alpha_{F(\zeta)} (\xi))$
is strongly convex, with a minimum at $\xi=0$ and no other critical points. 
In view of (\ref{eq:lower-estimate}) the set
\begin{equation}\label{eq:Dzeta}
	\Dscr_\zeta := \{\xi\in \d :  
	\rho(F(\zeta) + \alpha_{F(\zeta)} (\xi)) < \rho(F(\zeta)) + \tilde \epsilon(\zeta)\}
\end{equation}
contains the origin, is simply connected (a disc), and is compactly contained in $\d$;
furthermore, the discs $\Dscr_\zeta$ depend smoothly on the point $\zeta\in I$.  
Choose a smooth family of diffeomorphisms  $\phi_\zeta\colon \cd\to \overline \Dscr_\zeta$ 
$(\zeta\in I)$ which are holomorphic in $\d$ and satisfy $\phi_\zeta(0)=0$. 
Let $\alpha\colon I\times\cd\to \r^3$ be defined by
\begin{equation}\label{eq:alpha}
	\alpha(\zeta,\xi) = \alpha_{F(\zeta)} (\phi_\zeta(\xi)),\quad \zeta\in I,\  \xi\in\cd.
\end{equation}
Pick a smooth function $r\colon I\to [0,1]$ such that $r(\zeta)=1$ when $\epsilon(\zeta)\ge \delta/2$ 
and the support of $r$ is contained in the relative interior of $J_0\cap U$.

We now apply Theorem \ref{th:RH} to the conformal minimal immersion $F\colon M\to D$,
the map $\alpha$ given by (\ref{eq:alpha}), and the function $r$. It is straightforward to verify that the 
resulting conformal minimal immersion $G\colon M\to D$ 
satisfies the conclusion of Proposition \ref{prop:lifting} provided that the number $\eta>0$ 
in  Theorem \ref{th:RH} is chosen small enough.
The existence of a constant $C_0>0$ satisfying the estimate {\em (4)} in Proposition \ref{prop:lifting} 
is immediate from the geometry of the discs $\alpha_\bx^j(\cdotp)$  furnished by 
Lemma \ref{lem:M-discs}. Indeed, we clearly have a uniform estimate $\|\alpha_\bx^j(\xi)\|\le b|\xi|$ 
$(\xi\in\cd, \ \bx\in L,\ j=1,2)$  for some constant $b>0$.
From (\ref{eq:alpha}), we get $\|\alpha(\zeta,\xi)\| \leq  b|\phi_\zeta(\xi)|$ for $\zeta\in I$ and $\xi\in\cd$.
Together with \eqref{eq:estimate-c}, \eqref{eq:Dzeta}, and \eqref{eq:alpha} one obtains
\[	
 \epsilon_0\ge \tilde\epsilon(\zeta) \ge 
 \rho(F(\zeta)+\alpha(\zeta,\xi)) - \rho(F(\zeta)) \geq c|\phi_\zeta(\xi)|^2 \geq c/b^2 \|\alpha(\zeta,\xi)\|^2
\]
which gives $\|\alpha(\zeta,\xi)\| \le C_0 \sqrt{\epsilon_0}$  with $C_0=b/\sqrt{c}$. 
By increasing $C_0$ slightly, this gives {\em (4)} provided that 
the approximation in Theorem \ref{th:RH}  (see (\ref{eq:varkappa}) and {\it (i)}) is close enough. 

If the support of the function $\epsilon$ contains a boundary curve of $M$, then we write 
$\epsilon=\epsilon_1 + \epsilon_2$ where each of the two nonnegative functions 
$\epsilon_1,\epsilon_2 \colon bM\to  [0,\epsilon_0]$ satisfies the conditions of the special
case considered above. By first deforming $F$ to $G_1$ using the function $\epsilon_1$,
and subsequently deforming $G_1$ to $G=G_2$ using the function $\epsilon_2$,
the resulting conformal minimal immersion $G$ satisfies the conclusion of Proposition \ref{prop:lifting}, 
provided that the approximations are sufficiently close at each step.
\end{proof}

%
%
%
%
We now explain how to avoid critical points of a Morse exhaustion function $\rho\colon D\to\r$ when applying 
Proposition \ref{prop:lifting}. To this end, we adapt the method from \cite[Section 3.11]{Forstneric2017book}.

\begin{definition} A critical point $\bx_0$ of a $\Cscr^2$ function $\rho$ is {\em nice}
if,  in some neighborhood of $\bx_0$, $\rho$ agrees with its second order Taylor polynomial
at $\bx_0$.
\end{definition}

\begin{lemma}\label{lem:nice}
Every Morse function $\rho$ can be approximated arbitrarily closely in the fine $\Cscr^2$ topology
by a Morse function $\wt \rho$ with the same critical locus and with nice critical points. 
Furthermore, $\wt\rho$ can be chosen to agree with $\rho$ outside an arbitrarily small 
neighborhood of the critical locus.
\end{lemma}

\begin{proof}
Assume that $\bx_0$ is an (isolated) critical point of $\rho$ and 
\[
	\rho(\bx) = Q(\bx) + \eta(\bx),\quad 
	\lim_{\bx\to\bx_0} \frac{\eta(\bx)}{\|\bx-\bx_0\|^2} = 0.
\]
Choose a smooth increasing function $\chi\colon \r\to [0,1]$ such that $\chi(t)=0$ for $t\le 1$
and $\chi(t)=1$ for $t\ge 2$. Given $\epsilon>0$, we consider the function
\[
	\rho_\epsilon(\bx) =  Q(\bx)  + \chi(\epsilon^{-1}\|\bx-\bx_0\|) \, \eta(\bx).
\]
Then $\rho_\epsilon=Q$ on the ball $\|\bx-\bx_0\|\le \epsilon$ and $\rho_\epsilon=\rho$
on  $\|\bx-\bx_0\|\ge 2\epsilon$. As $\epsilon\to 0$, the $\Cscr^2$ norm of 
$\rho(\bx) - \rho_\epsilon(\bx) = (1-\chi(\epsilon^{-1}\|\bx-\bx_0\|))  \, \eta(\bx)$ tends to zero.
If $\epsilon>0$ is chosen small enough, then $\rho_\epsilon$ satisfies the conclusion of the lemma
at the critical point $\bx_0$. The same modification can be performed simultaneously at all 
critical points of $\rho$.
\end{proof}

A minimal strongly plurisubharmonic function has no critical points of index greater than $1$ (see Remark \ref{rem:homotopy}).
Critical points of index zero are local minima and are not approached by the boundary $F(bM)$ 
when applying Proposition \ref{prop:lifting}.

Assume now that $\bx_0$ is a nice Morse critical point  of $\rho$ with Morse index $1$. 
The subsequent analysis is local near $\bx_0$, so we may assume, after a rigid motion of $\r^3$,
that $\bx_0=\zero\in\r^3$, $\rho(\bx_0)=0$, and
\begin{equation}\label{eq:rho}
	\rho(\bx)=\rho(x_1,x_2,x_3) =   - a_1 x_1^2 + a_2 x_2^2 + a_3 x_3^2 + \eta(\bx),
\end{equation}
where $-a_1<0<a_2\le a_3$ and the function $\eta$ vanishes in a neighborhood of the origin. 
Note that $a_1<a_2$ since $\rho$ is minimal strongly  plurisubharmonic. 
Choose a number $c_0 >0$ small enough such that $\eta$ vanishes on the set
\begin{equation}\label{eq:P0}
	P_{c_0} := \{(x_1,x_2,x_3) \in \r^3: a_1 x_1^2 \le c_0,\ a_2 x_2^2+a_3 x_3^2 \le 4c_0\}.
\end{equation}
The straight line arc $E\subset \r^3$, defined by
\begin{equation}\label{eq:E}
	E=\{(x_1,0,0)\in \r^3: a_1 x_1^2 \le c_0\},
\end{equation}
 is a local stable manifold of the critical point $\zero$ of $\rho$. 
 Set $\lambda=a_2/a_1 >1$.  Choose a number $\mu\in\r$ with $1<\mu<\lambda$ and set  
\begin{equation}\label{eq:t0}
 	t_0=c_0(1-1/\mu)^2;
\end{equation}
hence $0<t_0<c_0(1-1/\lambda)^2 <c_0$. 

The following is \cite[Lemma 3.11.1, p.\ 98]{Forstneric2017book}, adapted to the situation at hand.

%
%
%
%
\begin{lemma} \label{lem:passing-model}
(Assumptions as above.) Assume that $0$ is the only critical value of the function $\rho$ (\ref{eq:rho}) 
in the set $\{-c_0 < \rho <3c_0\}$. Then there exists a minimal strongly plurisubharmonic function 
$\tau\colon D\cap \{\rho<3c_0\}\to \r$ satisfying the following conditions: 
\begin{itemize}
\item[\rm (a)]  $\{\rho \le -c_0\} \cup E \subset \{\tau\le 0\} \subset \{\rho \le -t_0\}\cup E$
(here, $E$ is given by (\ref{eq:E}));
\item[\rm (b)]  $\{\rho \le c_0\} \subset \{\tau \le 2c_0\} \subset \{\rho < 3c_0\}$;
\item[\rm (c)]  there is a constant $t_1\in (t_0,c_0)$ such that $\tau=\rho +t_1$ outside 
the set $P_{c_0}$ (\ref{eq:P0});
\item[\rm (d)]  $\tau$ has no critical values in the interval $(0,2c_0]$. 
\end{itemize}
\end{lemma}

The sublevel sets $\{\tau<c\}$ for $c>0$ in a neighborhood of the origin are shown in 
\cite[Figure 3.5, p.\ 100]{Forstneric2017book} (in a similar setting of strongly plurisubharmonic functions).

\begin{proof}
The choice of the number $t_0$ \eqref{eq:t0} implies that there is a smooth convex increasing function  
$h\colon \r \to [0,+\infty)$ satisfying the following conditions:
\begin{itemize}
\item[(i)]   $h(t)=0$ for $t\le t_0$;
\item[(ii)]   for $t\ge c_0$ we have $h(t)=t - t_1$ with $t_1=c_0 - h(c_0) \in (t_0,c_0)$;
\item[(iii)]  for $t_0\le t\le c_0$ we have $t-t_1 \le h(t) \le t - t_0$; 
\item[(iv)]  for all $t\in\r$ we have that $0\le \dot h(t) \le 1$ and $2t\ddot h(t) + \dot h(t) < \lambda$.
\end{itemize}
The construction of such function is entirely elementary (cf.\ \cite[pp.\ 98-99]{Forstneric2017book}; 
its graph is shown on \cite[Fig.\ 3.4, p.\ 99]{Forstneric2017book}).      
Let $\tau\colon\r^3\to\r$ be given by
\begin{equation}
\label{eqn:tau}
    \tau(\bx) = - h(a_1 x_1^2) + a_2 x_2^2 + a_3 x_3^2 + \eta(\bx).
\end{equation}
Setting $t=a_1 x_1^2$, a calculation shows that on the set $P_{c_0} \subset \{\eta=0\}$  (\ref{eq:P0}) we have
\[
	- \frac{\partial ^2\tau(\bx)}{\partial x_1^2} = 2a_1 \left( 2t\ddot h(t) + \dot h(t) \right) 
	< 2a_2 = \frac{\partial^2 \tau(\bx)}{\partial x_2^2},
\]
where the inequality holds by property (iv) of $h$ (recall that $\lambda=a_2/a_1$). 
This shows that $\tau$ is minimal strongly plurisubharmonic on $P_{c_0}$.  The other properties of 
$\tau$ follow immediately from the properties of $h$.  
(Compare with the proof of \cite[Lemma 3.11.1, p.\ 98]{Forstneric2017book}.)  
Condition (c) shows that $\tau$ is minimal strongly plurisubharmonic 
also on the complement of $P_{c_0}$. Condition (d) obviously holds on $P_{c_0}$,
while on the complement of $P_{c_0}$ it follows from  (c) and the assumptions on $\rho$.
\end{proof}

Combining Proposition \ref{prop:lifting} and Lemma \ref{lem:passing-model}, we now prove 
the following lemma which provides the induction step in the proof of Theorem \ref{th:main1}.

%
%
%
%
\begin{lemma}\label{lem:lifting2}
Let $\rho$ be a minimal strongly plurisubharmonic function on a domain $D\subset\r^3$,
and let $a<b$ be real numbers such that the set
\begin{equation}\label{eq:Dab}
	D_{a,b}=\{\bx\in D: a<\rho(\bx) <b\}  
\end{equation}
is relatively compact in $D$. Given numbers $0<\eta<b-a$, $\epsilon>0$, $\delta>0$, a conformal minimal
immersion $F \colon M\to D$ such that $F(bM) \subset D_{a,b}$, a point $p_0\in \mathring M$, a number $d>0$,
and a compact set $K\subset \mathring M$, there exists a conformal minimal immersion $G\colon M \to D$
satisfying the following conditions: 
\begin{itemize}
\item[\rm (a)]   $G(bM) \subset D_{b-\eta,b}$ (equivalently, $b-\eta < \rho(G(\zeta)) <b$
for every $\zeta\in bM$);
\item[\rm (b)]   $\rho(G(\zeta))\ge \rho(F(\zeta)) -\delta$ for every $\zeta\in M$;
\item[\rm (c)]   $\|G-F\|_{1,K}<\epsilon$;
\item[\rm (d)]   $\dist_{G}(p_0,bM) > d$; 
\item[\rm (e)]   $\Flux(G)=\Flux(F)$.
\end{itemize}
\end{lemma}

\begin{proof}
If the domain $D_{a,b}$ (\ref{eq:Dab}) does not contain any critical points of $\rho$, then a finite number of 
applications of  Proposition \ref{prop:lifting} furnishes a conformal minimal immersion $G\colon M \to D$ satisfying 
all conditions except (d); this last condition can be achieved by an arbitrarily 
$\Cscr^0$ small deformation of $G$, using \cite[Lemma 4.1]{AlarconDrinovecForstnericLopez2015}. 
(The cited lemma allows one to increase the interior boundary distance of a conformal minimal immersion
by an arbitrarily big amount, while staying arbitrarily $\Cscr^0$-close to the given map.)

Assume now that $\bx_1,\ldots,\bx_m$ are the (nice) critical points of $\rho$ in $D_{a,b}$ (\ref{eq:Dab}).
We may assume that the numbers $c_j=\rho(\bx_j)$ are distinct, and we 
enumerate the points so that $a< c_1 < c_2<\cdots <c_m<b$. We may also assume that $\min_{bM} \rho\circ F \le c_1$,
since otherwise $a$ may be replaced by a constant satisfying $c_1 < a <\min_{bM} \rho\circ F$.

Pick $c_0>0$ such that the conclusion of Lemma \ref{lem:passing-model}
applies to the critical point $\bx_1$ of $\rho$ and the constant $c_0$.  
Applying Proposition \ref{prop:lifting} finitely many times, we can replace $F$ by a 
conformal minimal immersion $F_1\colon M\to D$ such that $F_1(bM) \subset D_{c_1-c_0,b}$
and $F_1$ satisfies conditions (b), (c) and (e) in Lemma \ref{lem:lifting2}
(with $F_1$ in place of $G$, and for some new constants $\epsilon_1$ and $\delta_1$
in place of $\epsilon$ and $\delta$). By general position, we can assume that $F_1(bM)$ avoids the local 
stable manifold $E$  (see (\ref{eq:E})) of the point $\bx_1$. 
Let $\tau$ be the function furnished by Lemma \ref{lem:passing-model}
(for the point $\bx_1$ and the constant $c_0$).
Applying Proposition \ref{prop:lifting} with the function $\tau$ finitely many times, we can lift
the boundary $F_1(bM)$ above the level $c_1=\rho(\bx_1)$ and thus obtain a new conformal minimal
immersion $G_1\colon M\to D$ satisfying $G_1(bM) \subset D_{c_1,b}$. 
As before, $G_1$ is chosen to satisfy conditions (b), (c) and (e) in Lemma  \ref{lem:lifting2}, with $G_1$
in place of $G$ and $F_1$ in place of $F$ (and for some new constants $\epsilon_2>0,\delta_2>0$). 

Now, we repeat the same procedure, first using Proposition \ref{prop:lifting} to push $G_1(bM)$ 
close to the level $\rho=c_2$, and subsequently lifting the boundary across 
$\rho=c_2$ by using Lemma \ref{lem:passing-model}. This furnishes a conformal 
minimal immersion $G_2\colon M\to D$ with $G_2(bM)\subset D_{c_2,b}$. 

In finitely many steps of this kind we find a conformal minimal immersion $G\colon M\to D$
satisfying $G(bM)\subset D_{b-\eta,b}$ (condition (a)) and condition (e). 
Since the number of steps depends only  on the geometry of $\rho$, 
we can fulfil conditions (b) and (c) by choosing the corresponding numbers $\epsilon_j>0$ 
and $\delta_j>0$ sufficiently small  at  every step. Finally, condition (d) is achieved as in the special
case by appealing to  \cite[Lemma 4.1]{AlarconDrinovecForstnericLopez2015}.
\end{proof}

%
%
%
%
\begin{proof}[Proof of Theorem \ref{th:main1}]
Let $F_0\colon M\to D$ be a conformal minimal immersion and $K$ be
a compact set in $\mathring M$. Given $\epsilon>0$, we shall find a complete proper conformal minimal
immersion $F\colon \mathring M\to D$ satisfying 
$\|F-F_0\|_{0,K}= \sup_{\zeta\in K}\|F(\zeta)-F_0(\zeta)\|<\epsilon$.  
Such $F$ will be found as the limit $F=\lim_{j\to\infty} F_j$ of a sequence of conformal
minimal immersions $F_j\colon M\to D$ that will be constructed by an inductive application
of Lemma \ref{lem:lifting2}.

Choose a minimal strongly plurisubharmonic Morse exhaustion function $\rho\colon D\to\r$.
Let $P=\{\bx_1,\bx_2,\ldots\}\subset D$ be the (discrete) critical locus of $\rho$, where the 
points $\bx_j$ are enumerated so that $\rho(\bx_1)<\rho(\bx_2)<\cdots$.
By Lemma \ref{lem:nice}, we may assume that every $\bx_j$ is a nice critical point of $\rho$.
Pick increasing sequences  $a_1<a_2<a_3\ldots$ and $d_1<d_2<d_3\ldots$ such that 
$\sup_M \rho\circ F_0 < a_1$, $\lim_{j\to\infty} a_j=+\infty$, and 
$\lim_{j\to\infty} d_j=+\infty$. Also, choose a decreasing sequence $\delta_j>0$ with 
$\delta=\sum_{j=1}^\infty \delta_j <\infty$. Fix a point $p_0\in \mathring K$.
We shall construct a sequence of smooth conformal minimal immersions
$F_j\colon M\to D$, an increasing sequence of compacts 
$K=K_0\subset K_1\subset  \cdots\subset \bigcup_{j=1}^\infty K_j=\mathring M$,
and  a decreasing sequence of positive numbers $\epsilon_j>0$ such that the following 
conditions hold for every $j=1,2,\ldots$:
\begin{itemize}
\item[\rm (i$_j$)]     $a_j < \rho\circ F_{j} < a_{j+1}$ on $M\setminus K_j$;  
\item[\rm (ii$_j$)]   $\rho\circ F_j > \rho\circ F_{j-1}-\delta_j$ on $M$;           
\item[\rm (iii$_j$)]   $\|F_{j}-F_{j-1}\|_{1,K_{j-1}} <\epsilon_j$;                        
\item[\rm (iv$_j$)]   $\dist_{F_j}(p_0,M\setminus K_j) > d_j$;                         
\item[\rm (v$_j$)]    $\Flux(F_j)=\Flux(F_{j-1})$;                                              
\item[\rm (vi$_j$)]   
$\epsilon_{j} < 2^{-1} \min\{\epsilon_{j-1}, \dist(F_{j-1}(M),bD), \inf_{\zeta\in K_{j-1}}\|dF_{j-1}(\zeta)\|  \}$.                                                     														                   
\end{itemize}
To begin the induction, set $\epsilon_0=\epsilon/2$ and $K=K_0$.
Assume inductively that, for some $j\in \n$, we have  found maps $F_0,\ldots, F_{j-1}$,
numbers $\epsilon_0,\ldots,\epsilon_{j-1}$, and compact sets $K_0,\ldots, K_{j-1}$
such that the above properties hold. Pick a number $\epsilon_j >0$ satisfying
condition (vi$_j$). Applying Lemma \ref{lem:lifting2} with the data $(F_{j-1}, K_{j-1}, \epsilon_j, d_j)$ 
furnishes a conformal minimal immersion $F_j\colon M\to D$ satisfying condition (i$_j$) on the boundary $bM$, 
conditions (ii$_j$), (iii$_j$), (v$_j$), and such that $\dist_{F_j}(p_0,bM) > d_j$.
Next, pick a compact set $K_j\subset \mathring M$ such that $K_{j-1}\subset \mathring K_j$
and conditions (i$_j$) and (iv$_j$) hold. (It suffices to take $K_j$ big enough.) 
This completes the induction step.

Condition (vi$_j$) implies that  $\sum_{k=j+1}^\infty \epsilon_k < \epsilon_{j}$ for every $j=0,1,\ldots$;
in particular, $\sum_{k=0}^\infty \epsilon_k < 2\epsilon_0=\epsilon$.
Condition (iii$_j$) ensures that the sequence 
$F_j$ converges uniformly on compacts in $\bigcup_{j=1}^\infty K_j=\mathring M$
to a harmonic map $F = \lim_{j\to\infty} F_j \colon \mathring M\to \overline D$.
Conditions (iii$_j$) and (vi$_j$) show that for every $j=0,1,\ldots$ we have that
\begin{equation}\label{eq:estimateFj}
	\|F- F_j\|_{1,K_j} \le \sum_{k=j}^\infty  \|F_{k+1}-F_k\|_{1,K_j}  < 
	\sum_{k=j}^\infty \epsilon_{k+1} < 2\epsilon_{j+1} < \epsilon_j.
\end{equation}
In particular, $\|F-F_0\|_{0,K}<\epsilon$. The estimate \eqref{eq:estimateFj}, together with (vi$_{j+1}$), also shows that 
$F(K_j)\subset D$; since this holds for all $j$, we have $F(\mathring M)\subset D$. 
Since $2\epsilon_{j+1} < \inf_{\zeta\in K_{j}}\|dF_{j}(\zeta)\|$ by  (vi$_{j+1}$), it follows from (\ref{eq:estimateFj}) 
that $F$ is a conformal immersion on $K_j$. As this holds for all $j$, 
$F \colon \mathring M\to D$ is a conformal harmonic (hence minimal) immersion.
In view of (v), we have $\Flux(F)=\Flux(F_0)$ . Finally, conditions (i$_j$)--(iii$_j$) ensure that $F$ is proper into $D$, 
while conditions (iii$_j$) and (iv$_j$) show that $F$ is complete. 
\end{proof}

\begin{proof}[Proof of Theorem \ref{th:main1-bis}]
The proof is the same as that of Theorem \ref{th:main1} modulo the obvious modifications, replacing conditions pertaining
to the distance from $bD$ (see condition (vi$_j$) above) by the corresponding conditions pertaining to the distance from the
end of the domain $\Omega$ on which the function $\rho$ tends to $+\infty$.
\end{proof}

\begin{proof}[Proof of Theorem \ref{th:main2}]
Choose a minimal strongly  plurisubharmonic function $\rho$ on an open set  
$D'\supset \overline D$ such that $D=\{\bx \in D': \rho(\bx)<0\}$ and $d\rho\ne 0$ on $bD=\{\rho=0\}$.
Pick $\eta>0$ such that the set $\{\rho<\eta\}$ is relatively compact in $D'$ and 
$d\rho\ne 0$ on  the compact set 
\begin{equation}\label{eq:L}
	L=\{\bx \in D':  -\eta\le \rho(\bx)\le \eta\}.
\end{equation} 
Let $C_0>0$ be a constant satisfying the conclusion of Proposition  \ref{prop:lifting} for the data $(D',\rho,L)$.
In view of Theorem \ref{th:main1}, we may assume that the given 
conformal minimal immersion $F_0\colon M\to D$ satisfies 
\begin{equation}\label{eq:a0}
	a_0=a_0(F_0) :=\inf_{\zeta\in bM}\rho(F_0(\zeta))>-\eta.
\end{equation} 
(Equivalently, $F_0(bM)\subset D\cap \mathring L$.)
For every $j=0,1,2,\ldots$ we set  
\[
	a_j=2^{-j}a_0, \quad  \eta_j=a_{j+1}-a_{j} = 2^{-j-1} |a_0|.
\]
Pick an increasing sequence
$0<d_1<d_2<\cdots$ with  $\lim_{j\to\infty} d_j=+\infty$ and a decreasing sequence $\delta_j>0$ with 
$\delta=\sum_{j=1}^\infty \delta_j <\infty$.  
By following the proof of Theorem \ref{th:main1}, using also the estimate {\em (4)}
in Proposition \ref{prop:lifting} with the constant $C_0$ introduced above,
we find a sequence of conformal minimal immersions
$F_j\colon M\to D$ $(j=1,2,\ldots)$, an increasing sequence of compacts 
$K=K_0\subset K_1\subset  \cdots\subset \bigcup_{j=1}^\infty K_j=\mathring M$,
and  a decreasing sequence of numbers $\epsilon_j>0$ such that the following 
conditions hold for all $j=1,2,\ldots$:
\begin{itemize}
\item[\rm (i$_j$)]    $\rho\circ F_j>a_j$ on $M\setminus K_j$;
\item[\rm (ii$_j$)]   $\rho\circ F_j > \rho\circ F_{j-1}-\delta_j$ on $M$;
\item[\rm (iii$_j$)]  $\|F_{j}-F_{j-1}\|_{1,K_{j-1}} < \epsilon_j$;
\item[\rm (iv$_j$)]   $\dist_{F_j}(p_0,M\setminus K_j) > d_j$;
\item[\rm (v$_j$)]    $\Flux(F_j)=\Flux(F_{j-1})$; 
\item[\rm (vi$_j$)]   $\epsilon_{j} < 2^{-1} \min\{\epsilon_{j-1}, \dist(F_{j-1}(M),bD),   
\inf_{\zeta\in K_{j-1}}\|dF_{j-1}(\zeta)\| \}$;
\item[\rm (vii$_j$)]  $\|F_j-F_{j-1}\|_{0,M} \le C_0\sqrt{2\eta_j}= C_0\sqrt{2^{-j}}\sqrt{|a_0|}$. 
\end{itemize}
These properties correspond to those in the proof of Theorem \ref{th:main1},
except that condition (i$_j$) is adjusted to the present setting, and the additional condition  (vii$_j$)
follows from the estimate {\em (4)} in Proposition \ref{prop:lifting}.
(By \cite[Lemma 4.1]{AlarconDrinovecForstnericLopez2015}, condition (iv$_j$) can be achieved 
by a deformation which is arbitrarily small in the $\Cscr^0(M)$ norm, and the error made 
by this deformation is absorbed by the constant $C_0$ in (vii$_j$).) 

Set $C_1= C_0 \sum_{j=1}^\infty \sqrt{2^{-j}}$. Condition 
(vii$_j$) ensures that the sequence $F_j$ converges uniformly on $M$ to a continuous map
$F\colon M\to \overline D$ satisfying $\|F-F_{0}\|_{0,M} \le C_1\sqrt{|a_0|}$.
On the set $L$  \eqref{eq:L} the function $|\rho|$ is proportional to the distance from $bD$, so 
the number $|a_0|$, defined by (\ref{eq:a0}), is proportional to $\max_{\zeta\in bM} \dist(F_0(\zeta),bD)$.
This gives the estimate (\ref{eq:root}) in Theorem  \ref{th:main2}  for a suitable choice of 
the constant $C>0$ which depends only on the geometry of  $\rho$ in $L$. 
We can see as in the proof of  Theorem  \ref{th:main1} that $F|_{\mathring M}\colon \mathring M\to D$ 
is a proper complete conformal minimal immersion. 
\end{proof}

\begin{proof}[Proof of Corollary \ref{co:inf-top}] 
This follows from Lemma \ref{lem:lifting2} by a similar inductive procedure as those  in the proofs 
of Theorems \ref{th:main1} and  \ref{th:main1-bis}. In this case we use in addition
the Mergelyan approximation theorem for conformal minimal immersions 
(cf.\ \cite{AlarconForstnericLopez2016MZ, AlarconLopez2012JDG}) 
in order to add either a handle or an end to the surface at each step in the recursive construction. 
In this way, we may prescribe the topology of the limit surface. For the details of this construction,
we refer to the proof of Theorem 1.4 {\rm (b)}  and Corollary 1.5 {\rm (b)} 
in \cite{AlarconDrinovecForstnericLopez2015}.
\end{proof}

%
%

\begin{remark}\label{rem:generalizations}
The methods developed in \cite{AlarconDrinovecForstnericLopez2015} and in this paper 
allow us to generalize Theorems \ref{th:main1} and \ref{th:main2} to 
$(n-2)$-convex domains $D\subset \r^n$ for any $n>3$. We shall not state these generalizations,
but will give a brief sketch of proof. By definition, such a domain 
admits a smooth strongly $(n-2)$-plurisubharmonic exhaustion function
$\rho\colon D\to \r$ (see Definition  \ref{def:p-hull} and Proposition \ref{prop:p-convex}).
Furthermore, by convexifying in the normal direction, $\rho$ can be chosen such that the 
level sets $S_c=\{\rho=c\}$ for noncritical values of $\rho$ are strongly $(n-2)$-convex hypersurfaces, 
which means in particular that at every point $p_0 \in S_c$ there is a $2$-dimensional plane $L\subset T_{p_0} S_c$
on which $\Hess_\rho$ is strongly positive. (See Definition \ref{def:SMC}.)
By choosing suitably shaped small flat discs $\Delta_p\subset \r^n$ for points $p$ near $p_0$, 
lying in affine $2$-planes parallel to $L$, and solving the associated Riemann-Hilbert boundary value problem
(see \cite[Theorem 3.6]{AlarconDrinovecForstnericLopez2015}), one can lift a small part of the boundary 
of any conformal minimal disc $F\colon \cd\to D$ in a neighborhood of $p_0$ to a higher level set
of $\rho$ (see Proposition \ref{prop:lifting}). The rest of the proof goes through as before.
However, if the level sets of $\rho$ are merely $(n-1)$-convex (which is the same as mean-convex), 
this approach would require the existence of approximate solutions of the Riemann-Hilbert boundary 
value problem for nonflat conformal minimal discs (i.e., the analogue of 
\cite[Lemma 3.1]{AlarconDrinovecForstnericLopez2015} for $n>3$). 
We are unable to prove optimal results for $n>3$ at this time.

The corresponding optimal results in complex analysis, pertaining to the existence of
proper holomorphic maps from strongly pseudoconvex Stein domains to $q$-convex manifolds, 
were obtained in the papers \cite{DrinovecForstneric2007DMJ,DrinovecForstneric2010AJM}.
\qed\end{remark}


\section{Maximal minimally convex domains and a Maximum Principle at infinity}\label{sec:maximal}

This section is devoted to the proof of Theorem \ref{th:FTC}. 
The main ingredient is the maximum principle for minimal surfaces 
with finite total curvature in minimally convex domains in $\r^3$,
given by the following theorem.

\begin{theorem} \label{th:FTC-infinity} 
Assume that $S\subset \r^3$ is a complete, connected, immersed minimal surface with compact boundary $bS\neq\emptyset$ 
and finite total curvature. If  $D\subset\r^3$ is a smoothly bounded
minimally convex domain containing $S$, then 
$\dist(S,\r^3\setminus D)=\dist(bS,\r^3\setminus D)$.
\end{theorem}

The particular case of Theorem \ref{th:FTC-infinity} when $S$ is compact (with boundary) 
is already ensured by Proposition \ref{prop:max-principle}. The main difficulty in the general case 
(when the surface $S$ is not compact) is that one must deal with the contact at infinity, a rather delicate task.

Before proving Theorem \ref{th:FTC-infinity}, we show how it implies Theorem \ref{th:FTC} by a 
Kontinuit\"atssatz type argument (see Proposition \ref{prop:Kontinuitatssatz}).

\begin{proof}[Proof of Theorem \ref{th:FTC}, assuming Theorem \ref{th:FTC-infinity}] 
Assume that $D^c=\r^3\setminus D \neq\emptyset$ and let us prove that $S$ is a plane. 
Choose a relatively compact disc $\Omega\subset S$ and set 
$S'=S\setminus\Omega$. By Theorem \ref{th:FTC-infinity} we have that $\dist(S',D^c)=\dist(bS',D^c)$, and hence 
$\dist(S,D^c)=\dist(\overline\Omega,D^c)$. Thus, there exist points $\bx_0\in S$ and $\by_0\in bD$ 
such that $\|\bx_0-\by_0\|=\dist(S,D^c)$, and we infer from 
Proposition \ref{prop:max-principle} that $S$ is a plane.
Without loss of generality we may assume that $S=\{(x,y,z)\in\r^3\colon z=0\}$.
Set $W_+=\{(x,y,z)\in\r^3\colon z>0\}$. We claim that, if the set $W_+\setminus D$ is nonempty, 
then it is a halfspace. Indeed, assume that $W_+\setminus D\neq \emptyset$ and let us prove first that 
\begin{equation}\label{eq:c_+}
  d_+:=\dist(S,bD \cap W_+)=\dist(S,W_+\setminus D)>0.
\end{equation}
Consider the family of vertical negative half-catenoids
\[
	\Sigma_a=\{(x,y,z)\in\r^3 : x^2+y^2=a^2\cosh^2(z/a),\, z\leq 0\},\quad 0<a\leq 1.
\]
Let $A_+$ denote the cylinder
\[
A_+ :=\{(x,y,z)\in\r^3 :  x^2+y^2\le 2,\ 0\leq z\leq \tau_+\},
\]
where $\tau_+>0$ is chosen small enough such that $((0,0,\tau_+)+\Sigma_1)\cap \{z\geq 0\}\subset A_+\subset D$. 
The  Kontinuit\"atssatz for minimal surfaces (cf.\ Proposition \ref{prop:Kontinuitatssatz}) implies that 
\[ 
	\Sigma_a^+ := ((0,0,\tau_+)+\Sigma_a)\cap \{z\geq 0\}\subset D\quad \text{for all $0<a\leq 1$}. 
\]
Indeed, $\Sigma_a^+$ are minimal surfaces with boundaries in $A_+\cup S\subset D$, and 
$\Sigma_1^+ \subset A_+\subset D$.  It is easily seen that $\bigcup_{0<a\le 1} \Sigma_a^+$ contains 
the set $\{(x,y,z)\in\r^3 : x^2+y^2\ge 2,\  0\leq z< \tau_+\}$. 
Since $A_+\subset D$, we infer that $D$ contains the slab $\{0\le z< \tau_+\}$. This implies that 
$d_+\geq\tau_+>0$, thereby proving \eqref{eq:c_+}.

If there is a point $(x_0,y_0,d_+)\in bD$, then, arguing as in the proof of Proposition \ref{prop:max-principle} and using that $D$ is connected,
we easily infer that the plane $\Pi_+:=\{z=d_+\}$ lies in $bD\cap\{z>0\}=b(W_+\setminus D)$, and hence $W_+\setminus D$ is a halfspace.
Otherwise, $\Pi_+\subset D$ and we may reason as above (replacing $S$ by $\Pi_+$) to see that 
$\dist(\Pi_+,W_+\setminus D)>0$ in contradiction to \eqref{eq:c_+}. 

A symmetric argument guarantees that $\{(x,y,z)\in\r^3 : z<0\}\setminus D$ is either empty or a halfspace. 
This concludes the proof.
\end{proof}


The proof of Theorem \ref{th:FTC-infinity} also follows from a Kontinuit\"atssatz argument; however, 
the construction of a suitable family of minimal surfaces is much more delicate. The surfaces 
will be multigraphs, obtained as solutions of suitable Dirichlet problems 
for the minimal surface equation over finite coverings of annuli in $\r^2$; see Lemma \ref{lem:williams}.

Before going into the construction, we introduce some notation. 

\begin{definition}\label{def:arr}
For each pair of numbers $0\leq R_0<R\leq +\infty$ we set 
\[
	A_{R_0,R}:=\{(x,y)\in \r^2 : R_0< \|(x,y)\| < R\},\quad A_{R_0}=A_{R_0,+\infty}.
\]  
Endow $A_{0}=\r^2\setminus\{(0,0)\}$ with the Euclidean metric and denote by 
\[
	\pi_n : A^n_{0}\to A_0
\]
the $n$-sheeted isometric covering, $n\in \n$. We also set:
\begin{itemize}
\item $A^n_{R_0,R}:=\pi_n^{-1} (A_{R_0,R})$ for all $0\leq R_0<R<+\infty$, and $A^n_{R_0}=\pi_n^{-1} (A_{R_0})$;
\item $c_R:=\{(x,y)\in \r^2 : \|(x,y)\|=R\}$ and  $c^n_R:=\pi_n^{-1}(c_R)$, $R>0$.
\end{itemize}
Obviously, $bA^n_{R_0}=c^n_{R_0}$ and $bA^n_{R_0,R}=c^n_{R_0}\cup c^n_R$, $0<R_0<R<+\infty$, $n\in \n$. 
\end{definition}

A function $u\in \Cscr^2(A^n_{R_0,R})$ is said to {\em satisfy the minimal surface equation in $A^n_{R_0,R}$} if 
\[
{\rm div}\Big(\frac{\nabla u}{\sqrt{1+\|\nabla u\|^2}}\Big)=0\quad \text{in $A^n_{R_0,R}$};
\]
equivalently, if   $\{(p,u(p)) : p\in A^n_{R_0,R}\}$ is a minimal surface (a minimal multigraph). 

Given $\phi \in \Cscr^0(bA^n_{R_0,R})$, a function $u\in \Cscr^2(A^n_{R_0,R})\cap  \Cscr^0(\overline{A}^n_{R_0,R})$ 
is said to be a {\em solution of the Dirichlet problem for the minimal surface equation in $A^n_{R_0,R}$ with boundary 
data $\phi$} if $u$ satisfies the minimal surface equation in $A^n_{R_0,R}$  and the boundary condition $u|_{bA^n_{R_0,R}} =\phi$.

\begin{lemma}\label{lem:williams}
Let $0<R_0<R_1$, $K\in (0,1)$, and $n\in \n$. There exists a number $\epsilon>0$,  depending only on $R_0$, $R_1$, and $K$,  
such that the following holds. If $R\geq R_1$, $\delta\in [0,\epsilon]$,
\begin{enumerate}[{\rm (a)}]
\item $v\colon \overline{A}^n_{R_0}\to \r$ is a real analytic solution of the minimal surface equation in $A^n_{R_0}$,
\item   $\|\nabla v\|<K/2$ in $\overline{A}^n_{R_0}$, and
\item we set  $\phi_{R,\delta}\colon bA^n_{R_0,R}\to \r$, $\phi_{R,\delta}=v$ in $c^n_{R_0}$ and $\phi_{R,\delta} =v+\delta$ in  $c^n_{R}$,
\end{enumerate}
then  the Dirichlet problem for the minimal surface equation in $A^n_{R_0,R}$ with boundary data $\phi_{R,\delta}$ 
has a unique solution $u_{R,\delta}$. Furthermore, $u_{R,\delta}$ enjoys the following conditions:
\begin{enumerate}[{\rm (i)}]
\item $v\leq u_{R,\delta}\leq v+\delta$ on $ \overline{A}^n_{R_0,R}$;  
\item $u_{R,\delta}$ depends continuously on $(R,\delta)\in [R_1,+\infty) \times [0,\epsilon]$; 
\item  $\{u_{R,\delta}\}_{R>R_1}\to v$ as $R\to +\infty$ on compact subsets of $ \overline{A}^n_{R_0}$ for all $\delta \in [0,\epsilon]$.
\end{enumerate}
\end{lemma}
The number $\epsilon>0$ in the lemma will only depend on the existence   of  suitable barrier functions $\nu_{p_0}$ at boundary points $p_0\in bA^n_{R_0,R}$ adapted to our problem. The construction of these barrier functions in turn only depends on 
the constants $R_0$, $R_1$, and  $K$.

\begin{proof}
For the existence part of the lemma, we use Perron's method for the minimal surface equation; see for instance \cite{GilbargTrudingerbook,Nitschebook}. 

Take arbitrary numbers $R> R_0>0$ and $\delta\geq 0$. If $w\in \Cscr^0(\overline{A}^n_{R_0,R})$, 
$D$ is a convex disc in $\overline{A}^n_{R_0,R}$, 
and  $w_D$ is  the solution of the minimal surface equation in $D$ which
equals $w$ on $bD$ (such exists by classical Rado's theorem), we denote by $\wh w_D$  the continuous function in 
$ \overline{A}^n_{R_0,R}$ which coincides with $w$ in  $\overline{A}^n_{R_0,R}\setminus D$ and with $w_D$  in $D$. 

By definition, a function  $w\in \Cscr^0(\overline{A}^n_{R_0,R})$ is said to be a  sub-solution of the Dirichlet problem 
for the minimal surface equation in $\overline{A}^n_{R_0,R}$, defined by  $\phi_{R,\delta}$ given in {\rm (c)}, if
$w\leq \phi_{R,\delta}$ in $bA^n_{R_0,R}$ and $w\leq \wh w_D$ in $D$ (hence in $\overline{A}^n_{R_0,R}$) 
for all discs $D$ as above. We denote by $\mathcal{F}^-_{R,\delta}$ the family of all such  sub-solutions. 
Likewise,  $w$ is said to be a  super-solution for this problem  if
$w\geq \phi_{R,\delta}$ in $bA^n_{R_0,R}$ and $w\geq \wh w_D$ in $D$ (hence in $\overline{A}^n_{R_0,R}$) for all discs  
$D$ in $\overline{A}^n_{R_0,R}$. The corresponding space of super-solutions will be denoted by $\mathcal{F}^+_{R,\delta}$. 
Note that  $v|_{\overline{A}^n_{R_0,R}}\in \mathcal{F}^-_{R,\delta}$ and  
$(v+\delta)|_{\overline{A}^n_{R_0,R}}\in \mathcal{F}^+_{R,\delta}$  for all $R>R_0$ and $\delta>0$, 
where $v$ is the function given in item {\rm (a)} in the statement of the lemma; hence these are nonempty families. 
If $w_1$ is a sub-solution and $w_2$ is a super-solution, then the maximum principle ensures that $w_1\leq w_2$. 
On the other hand, if $w_1$ and $w_2$ are sub-solutions (respectively, super-solutions), then $\max\{w_1,w_2\}$ 
(respectively, $\min\{w_1,w_2\}$) also is. 

We define
\begin{equation}\label{eq:uRd}
u_{R,\delta}\colon A^n_{R_0,R}\to \r, \quad u_{R,\delta}(p)=\sup_{w\in \mathcal{F}^-_{R,\delta}} w(p). 
\end{equation}
It is well known that $u_{R,\delta}$ is a solution of the minimal surface equation in $A^n_{R_0,R}$. Further, 
\begin{equation}\label{eq:abovebelow}
w_1\leq u_{R,\delta}\leq w_2\quad \text{for any $w_1\in\mathcal{F}^-_{R,\delta}$ and $w_2\in\mathcal{F}^+_{R,\delta}$.}
\end{equation}
In particular, 
\begin{equation} \label{eq:comparar}
v\leq u_{R,\delta} \leq v+\delta \quad \text{in $A^n_{R_0,R}$ for all $R>R_0$ and all $\delta\geq 0$.}
\end{equation}

\begin{claim}\label{cl:barrier}
Given numbers $R_1>R_0>0$ and $K \in (0,1)$, there exists $\epsilon>0$, depending only on $R_0$, $R_1$, and $K$, 
such that the following holds. 
If $R\geq R_1$ and $0\leq \delta\leq \epsilon$,  then the function $u_{R,\delta}$ given by \eqref{eq:uRd} is a solution 
to the Dirichlet problem for the minimal surface equation with boundary data $\phi_{R,\delta}$ in $A^n_{R_0,R}$; that is to say,
\[
\text{$\lim_{p\to p_0} u_{R,\delta}(p)=\phi_{R,\delta}(p_0)$\quad  for all $p_0\in bA^n_{R_0,R}$.}
\]
\end{claim}

\begin{proof}
Choose $R\geq R_1$ and a point $p_0\in bA^n_{R_0,R}$, and let us distinguish cases.

\noindent{\em Case 1: $p_0\in c_{R_0}^n$}. 

Let us prove the existence of a number $\epsilon_1>0$, 
depending only on $R_0$, $R_1$, and $K$, for which the following statement holds.
Given $\delta\in [0,\epsilon_1]$ there exists  $\nu_{p_0}\in \Cscr^0(\overline{A}^n_{R_0})$  such that 
$\nu_{p_0}(p_0)=v(p_0)=\phi_{R,\delta}(p_0)$ and   $\nu_{p_0}|_{\overline{A}^n_{R_0,R}}\in \mathcal{F}^+_{R,\delta}$. 

Indeed, write $\pi_n(p_0)=(x_0,y_0)\in c_{R_0}.$
Set $C_{R_0}:=\{(x,y,z)\in \r^3 : \|(x,y)\|<R_0\}$ and  
$J_K:=\{(x,y,z)\in \r^3 : 0\leq z-v(p_0)= K \|(x-x_0, y-y_0)\| \}$. 
Pick $\mu>0$ small enough in terms of $R_0$, $R_1$, and $K$, such that the set 
\[
\gamma:=\big((bC_{R_0} \cap J_K)\cap \{0\leq z-v(p_0)\leq \mu\}\big)\cup \big((J_K\setminus C_{R_0}) \cap \{z-v(p_0)=\mu\}\big)
\]
is a Jordan curve. It follows that $\gamma$ has one-to-one orthogonal projection $\gamma_0$ to the plane $\{z=0\}\equiv \r^2$. 
Ensure also that $\gamma_0$ bounds a topological disc  $U\subset \r^2$ with $\overline{U}\subset   A_{R_0,R_1}\cup c_{R_0}$. 
Thus, $\pi_n^{-1}(U)$ consists of $n$ disjoint isometric copies of $U$; write $\widehat U\subset A^n_{0}$ 
for the connected component of $\pi_n^{-1}(U)$ 
containing $p_0$. Denote by $\phi\colon b  U\to \r$ the (unique) continuous function such that 
$\{(p,\phi(p)) : p\in b U=\gamma_0\}=\gamma$.

Further, the domain $U$ satisfies an exterior sphere condition with radius $R_0$ (cf.\ \cite[Definition 1.4 (i)]{Williams1984JRAM}), 
and thus, if $\mu>0$ is sufficiently  small in terms of $R_0$, $R_1$,  and $K$, 
the Dirichlet problem for the minimal surface equation in $U$ with boundary data $\phi$ 
has a unique solution $f\colon U \to \r$ satisfying 
$(f\circ \pi_n)(p_0)=\phi(p_0)=v(p_0)$ and $f\circ \pi_n > v$ in $\widehat U\setminus\{p_0\}$; 
see \cite{Williams1984JRAM} and take into account condition {\rm (b)} in the statement of the 
lemma and that $\gamma\subset J_K$. It follows that 
\[
	\inf \{(f\circ \pi_n)(p)- v(p) : p\in b\widehat U\setminus c^n_{R_0}\}
	=\inf \{v(p_0)+\mu- v(p) : p\in b\widehat U\setminus c^n_{R_0}\}>0.
\] 
Finally, take $\epsilon_1>0$ smaller than this infimum. 
Note that this number does not depend on  $v$; take into account {\rm (b)}. Further, 
since $\mu$  depends  on $R_0$, $R_1$, and $K$ but not on  $p_0\in c^n_{R_0}$, the same holds for  $\epsilon_1$. It suffices to set $\nu_{p_0}\colon \overline{A}^n_{R_0}\to \r$, $\nu_{p_0}= \min\{f\circ \pi_n,v+\epsilon_1\}$ on $\widehat U$ and  $\nu_{p_0}= v+\epsilon_1$ on $\overline{A}^n_{R_0}\setminus \widehat U$. 

Since $\nu_{p_0}|_{\overline{A}^n_{R_0,R}}\in \mathcal{F}^+_{R,\delta}$, $\delta\in [0,\epsilon_1]$, and $\nu_{p_0}(p_0)=v(p_0)=\phi_{R,\delta}(p_0)$, the bounds \eqref{eq:abovebelow} and \eqref{eq:comparar} trivially ensure that $\lim_{p\to p_0} u_{R,\delta}(p)=\phi_{R,\delta}(p_0)$.

\smallskip

\noindent{\em Case 2: $p_0\in c_{R}^n$.} 

Let us now prove the existence of a number $\epsilon_2>0$, depending only on $R_0$, $R_1$, and $K$, 
for which the following statement holds. Given $\delta\in [0,\epsilon_2]$,  
 there exists  $\nu_{p_0}\in \Cscr^0(\overline{A}^n_{R_0,R})\cap  \mathcal{F}^-_{R,\delta}$ such that 
 $\nu_{p_0}(p_0)=v(p_0)+\delta=\phi_{R,\delta}(p_0)$.

Indeed, consider a small disc $V\subset \r^2$  centered at the origin and with radius less than $R_1-R_0$, 
set $U_R:=(\pi_n(p_0)+V) \cap \overline{A}_{R_0,R}$, and let $\widehat U_R\subset \overline{A}^n_{R_0,R}$ 
denote the connected component of $\pi_n^{-1}(U_R)$ containing $p_0$; obviously, 
$\pi_n|_{\widehat U_R} : \widehat U_R\to U_R$ is an isometry. Choose a linear function $f\colon \r^2 \to \r$ satisfying 
$(f\circ \pi_n)(p_0)=v(p_0)$ and $f\circ \pi_n < v$ on $\widehat U_R\setminus\{p_0\}$, and take 
$0<\epsilon_2<\inf \{v(p)-(f\circ \pi_n)(p) : p\in b\widehat U\setminus {c}^n_{R}\}$. 
The existence of such an  $\epsilon_2$ follows from item {\rm (b)}, and it can be chosen depending on neither   $v$ nor  $R$ 
(it only depends on $R_0$, $R_1$, and $K$). Given $\delta \in [0,\epsilon_2]$, it suffices to set 
$\nu_{p_0}\colon \overline{A}^n_{R_0,R}\to \r$, $\nu_{p_0}= \max\{f\circ \pi_n+\delta,v+\delta -\epsilon_2\}$ on 
$\widehat U_R$, and  $\nu_{p_0}= v+\delta-\epsilon_2$ on $\overline{A}^n_{R_0,R}\setminus \widehat U_R$.

As above, since $\nu_{p_0}\in \mathcal{F}^-_{R,\delta}$ and $\nu_{p_0}(p_0)=v(p_0)+\delta=\phi_{R,\delta}(p_0)$, \eqref{eq:abovebelow} and \eqref{eq:comparar} guarantee that $\lim_{p\to p_0} u_{R,\delta}(p)=\phi_{R,\delta}(p_0)$.

To finish the proof of the claim, it suffices to choose $\epsilon:=\min\{\epsilon_1,\epsilon_2\}$.
\end{proof}

We continue with the proof of Lemma \ref{lem:williams}. In view of Claim \ref{cl:barrier} it remains to check that, given 
numbers $\delta\in [0,\epsilon]$ and $R\geq R_1$, the solution $u_{R,\delta}$ given by \eqref{eq:uRd} is unique 
and satisfies conditions {\rm (i)}, {\rm (ii)}, and {\rm (iii)}. Uniqueness  follows directly from the maximum principle. 
Property {\rm (i)} is ensured by \eqref{eq:comparar}. Since the boundary data $\phi_{R,\delta}$ depend continuously 
on $(R,\delta)\in [R_1,+\infty) \times [0,\epsilon]$, the same holds for the solutions $u_{R,\delta}$, proving {\rm (ii)}.
In order to prove {\rm (iii)}, fix a number $\delta\in [0,\epsilon]$ and take any divergent sequence 
$\{R_j\}_{j\in \n}\subset [R_1,+\infty)$. 
By standard compactness results (see for instance \cite{Anderson1985ASENS}), the sequence $\{u_{R_j,\delta}\}_{j\in \n}$ 
converges uniformly on compact subsets of $\overline{A}^n_{R_0}$ to a solution $u$ of the minimal surface equation in 
$A^n_{R_0}$ with boundary data $u=v$ on $c^n_{R_0}$. Furthermore, \eqref{eq:comparar} gives that $v\leq u\leq v+\delta$, 
and hence $u=v$ by the maximum principle at infinity (see for instance \cite{MeeksRosenberg2008JDG}). 
This proves {\rm (iii)} and concludes the proof.
\end{proof}


\begin{proof}[Proof of Theorem \ref{th:FTC-infinity}]
If $S$ is compact, then the result follows from Proposition \ref{prop:max-principle}. 

Assume now that $S$ is not compact.  Up to passing to the two sheeted orientable covering, 
we may assume  that $S=X(M)$, where $M$ is a noncompact Riemann surface with compact boundary 
$bM\neq\emptyset$, and $X\colon M\to \r^3$ is  a complete conformal minimal immersion with finite total curvature. 
With this notation, we have $bS=X(bM)$. 

The assumptions on $S$ imply that $M$ is of finite topology and of parabolic conformal type (in particular, its ends are annular conformal punctures), and the (conformal) Gauss map ${\rm N}:M\to \s^2$ of $X$ extends conformally to the ends; see \cite{Ossermanbook}.  Given an annular end $E\subset M$, $E\cong \overline{\d}\setminus \{0\}$,  let ${\rm n}_E$ denote the limit normal vector of $X(E)$ at infinity and  $\Pi_E$  the vectorial plane in $\r^3$ orthogonal to ${\rm n}_E$. 

It is also well known that the minimal immersion  $X$ is a proper map and all the ends are {\em finite sublinear multigraphs}; see \cite{JorgeMeeks1983T}.  The latter means that, for any annular end $E$ of $M$, there exists  an open solid circular cylinder $C$, 
with axis parallel to  ${\rm n}_E$,  such that: 
\begin{enumerate}[\rm (i)]
\item $E\cap X^{-1}(\overline{C})$ is compact and contains $bE$;
\item $(\pi_E\circ X)|_{E\setminus X^{-1}(C)}\colon E\setminus X^{-1}(C)\to \Pi_E\setminus C$ is a finite covering, 
where $\pi_E\colon \r^3\to \Pi_E$ is the orthogonal projection;
\item $\lim_{j\to \infty} \frac1{\|X(p_j)\|} \langle {\rm n}_E,X(p_j)\rangle=0$ for any divergent sequence 
$\{p_j\}_{j\in \n}\subset E$.
\end{enumerate}
Write  $w_E$ for the winding number of $X(E)$ at infinity; note that $w_E$  is the degree of the covering 
$(\pi_E\circ X)|_{E\setminus X^{-1}(C)}$.

Recall that $bS\neq \emptyset$. If $D=\r^3$, there is nothing to prove.
Assume now that $D^c\neq \emptyset$ and let 
\[
	d:=\dist(S,D^c)<+\infty.
\] 
It suffices to prove the following claim:
\begin{equation}\label{eq:p0q0}
\text{there exist points $\bx_0\in S$ and $\by_0\in bD$ such that $\|\bx_0-\by_0\|=d$.}
\end{equation}
Indeed, assume for a moment that this holds.  If $\bx_0\in bS$, we are done. Otherwise, $\bx_0\in S\setminus bS$, 
and we infer from  Proposition \ref{prop:max-principle} that the surface
$S$ is flat and $\by_0-\bx_0+S\subset bD$. Thus, $d=\dist(bS,D^c)$ which concludes the proof of Theorem \ref{th:FTC-infinity} provided that \eqref{eq:p0q0} holds.

We now prove the assertion \eqref{eq:p0q0}. We reason by contradiction and assume that
\begin{equation}\label{eq:continf} 
\dist(\bx,D^c)>d\quad \text{for all $\bx\in S$}.
\end{equation}
Under this assumption, there exists an annular end $X(E)\subset S$ with $\dist(X(E),D^c)=d$; 
recall that $E$ is conformally equivalent to $\overline\d\setminus\{0\}$. Set 
\[
\text{$E_t:=t\, {\rm n}_E +X(E)$ \ for all $t\geq 0$.}
\]
In particular, $E_0=X(E)$. Condition \eqref{eq:continf} ensures that 
\begin{equation} \label{eq:region}
\bigcup_{t\in [0,d]} E_t\subset D.
\end{equation}
Set $C_R:=\{(x,y,z)\in \r^3 : \|(x,y)\|<R\}$ for $R>0$.  
Up to a rigid motion, a shrinking of $E$, and taking $R_0>0$ large enough,  
we may  assume that ${\rm n}_E=(0,0,-1)$,  $X(bE) \subset bC_{R_0}$ and 
\begin{equation}\label{eq:-1}
\dist(E_d,D^c)=0.
\end{equation}

Given $\delta\ge 0$, we set $\gamma_{R_0}(\delta):= b E_{d+\delta}=(d+\delta) {\rm n}_E+bE_0\subset bC_{R_0}$. 
Since $(\bigcup_{t\in [0,d]} E_t)\cap \overline C_{R}$ is compact for all $R\geq R_0$ (see {\rm (i)}), 
condition \eqref{eq:region} provides numbers $\epsilon>0$ and $R_1>R_0$ such that
\begin{equation}\label{eq:gammD1}
 \bigcup_{t\in [0,d+\epsilon]} (E_t\cap \overline{C}_{R_1})\subset D.
\end{equation}
In particular, $\gamma_{R_0}(\delta)\subset D$ for all $\delta\in[0,\epsilon]$. Set $\gamma_R(\delta):=E_{d+\delta} \cap bC_R$ 
for all $\delta\in [0,\epsilon]$ and $R>R_0$. For simplicity,  write $n$ for the winding number $w_E$ of $X(E)$ 
as multigraph over $A_{R_0}=\{(x,y)\in\r^2 : \|(x,y)\|>R_0\}$, and denote by $\phi_{R,\delta}\colon bA^n_{R_0,R}\to \r$ 
the unique analytic function satisfying
\[
\text{$\gamma_{R_0}(\delta)\cup \gamma_R(0)=\{(p,\phi_{R,\delta}(p)) : p\in bA^n_{R_0,R}\}$ \ 
for all $\delta\in [0,\epsilon]$ and $R>R_0$.}
\]
(See Definition \ref{def:arr} for the notation.)
Likewise, let $v\colon \overline{A}^n_{R_0}\to \r$ denote  the unique analytic function such that 
\[
\text{$E_d=\{(p,v(p)) : p\in \overline{A}^n_{R_0}\}$.} 
\]
Without loss of generality (increasing $R_0$ if necessary), we may assume that $\|\nabla v\|<1/4$ on $\overline{A}^n_{R_0}$; 
see Property {\rm (iii)} above.

On the other hand, if $\epsilon>0$ is chosen small enough, then Lemma \ref{lem:williams} shows that the Dirichlet problem for 
the minimal surface equation in $ A^n_{R_0,R}$ with boundary data $\phi_{R,\delta}$ has a unique solution 
 $u_{R,\delta}$ for all $R\geq R_1$ and $\delta \in [0,\epsilon]$. Furthermore, 
 \begin{equation}\label{eq:vduv}
 	v-\delta \leq u_{R,\delta}\leq v\ \  \text{in $ A^n_{R_0,R}$ for all $R\geq R_1$ and $\delta \in [0,\epsilon]$.}
 \end{equation}
 Fix a pair of numbers $\delta\in [0,\epsilon]$ and $R\geq R_1$, and set  
 $T_{R,\delta}:=\{(p,u_{R,\delta}(p)) : p\in \overline{A}^n_{R_0,R}\}$.  Note that \eqref{eq:gammD1} and 
\eqref{eq:vduv} guarantee that $T_{R_1,\delta}\subset D$, whereas \eqref{eq:region} and \eqref{eq:gammD1} ensure that 
$bT_{R,\delta}=\gamma_{R_0}(\delta)\cup \gamma_R(0)\subset D$ for all $R\geq R_1$. Thus, the  Kontinuit\"atssatz for 
minimal surfaces (Proposition \ref{prop:Kontinuitatssatz}) implies that $T_{R,\delta}\subset D$ for all $R\geq R_1$. 
Further, by Lemma \ref{lem:williams} we have $T_{R,\delta}\to E_{d+\delta}$ uniformly on compact subsets as $R\to+\infty$,
and hence  $E_{d+\delta}\subset \overline D$. Since this holds for arbitrary $\delta\in [0,\epsilon]$, we infer that 
$\bigcup_{t\in [0,d+\epsilon]} E_t\subset \overline D$, and hence $\bigcup_{t\in [0,d+\epsilon)} E_t\subset D$. 
This contradicts \eqref{eq:-1} and thereby proves \eqref{eq:p0q0}. 
\end{proof}


\section{Null hulls in $\c^n$ and minimal hulls in $\r^n$}\label{sec:hulls}

The approximate solution of the Riemann-Hilbert problem for null discs, furnished by 
\cite[Lemma 3.3]{AlarconDrinovecForstnericLopez2015},  allows us to  extend the main results of the paper 
\cite{DrinovecForstneric2015TAMS} to null hulls in $\c^n$, and  minimal hulls in $\r^n$, for any $n\ge 3$. 
We now explain these generalizations. The proofs are similar to those 
in \cite{DrinovecForstneric2015TAMS} and are left to the reader.

We denote  by $\Ngot(\d,\Omega)$ the set of all immersed null holomorphic discs $\cd\to\c^n$ with range in a 
domain $\Omega\subset \c^n$, and we write 
\[
	\Ngot(\d,\Omega,\bz ) =\{F\in \Ngot(\d,\Omega) : F(0)=\bz \}\quad
\text{for \ $\bz \in \Omega$}.
\]
The case $n=3$ of the next result is \cite[Theorem 2.10]{DrinovecForstneric2015TAMS}; 
the general case $n\ge 3$ is proved in exactly the same way 
by using  \cite[Lemma 3.3]{AlarconDrinovecForstnericLopez2015} instead of \cite[Lemma 2.8]{DrinovecForstneric2015TAMS}.

%
%
%
%

\begin{theorem} [Null plurisubharmonic minorant]
\label{th:nullpshminorant} 
Let $\phi\colon \Omega\to \r\cup\{-\infty\}$ be an upper semicontinuous 
function on a domain $\Omega\subset\c^n$ $(n\ge 3)$. Then the function 
\begin{equation}
\label{eq:nullPoisson-funct}
   u(\bz ) := \inf \Big\{\int^{2\pi}_0 \phi(F(\E^{\imath t}))\, 
   \frac{dt}{2\pi} : \ F\in \Ngot(\d,\Omega,\bz ) \Big\}, \quad \bz \in \Omega
\end{equation}
is null plurisubharmonic on $\Omega$ or identically $-\infty$; 
moreover, $u$ is the supremum of all null plurisubharmonic 
functions on $\Omega$ which are not greater than $\phi$. 
\end{theorem}

%
%
%
%
\begin{remark}\label{rem:basic}
The disc functional $P_\phi$ in (\ref{eq:nullPoisson-funct}),  
which assigns to any holomorphic disc $F\colon \cd\to\Omega$ 
the average $P_\phi(F)=\int^{2\pi}_0 \phi(F(\E^{\imath t}))\, \frac{dt}{2\pi} \in \r\cup\{-\infty\}$,  
is called the  {\em Poisson functional} associated to the function $\phi$.
If we take all holomorphic discs $F\colon \cd\to\Omega$ with $F(0)=\bz $ 
(instead of just null discs) in (\ref{eq:nullPoisson-funct}), we obtain 
the biggest plurisubharmonic function on $\Omega$ satisfying
$u\le \phi$. This fundamental result of Poletsky \cite{Poletsky1991PSPM,Poletsky1993IUMJ} and 
Bu and Schachermayer \cite{BuSchachermayer1992TAMS} was generalized  by Rosay \cite{Rosay2003IUMJ,Rosay2003JKMS} to all
complex manifolds $\Omega$ (see also L\'arusson and Sigurdsson \cite{LarussonSigurdsson1998JRAM,LarussonSigurdsson2003JRAM}),
and by Drinovec Drnov\v sek and Forstneri\v c \cite[Theorem 1.1]{DrinovecForstneric2012IUMJ}
to all locally irreducible complex space.
\qed\end{remark}

Given a domain $\omega\subset \r^n$, we denote by $\Mgot(\d,\omega)$ the set of all conformal minimal immersions
$\cd\to \omega$. 
By using Theorem \ref{th:nullpshminorant}, along with the connection between null plurisubharmonic and 
minimal plurisubharmonic functions  (see Lemma \ref{lem:minimal}), 
we obtain the following result. The case $n=3$ is \cite[Theorem 4.5]{DrinovecForstneric2015TAMS}.

%
%
%
%
\begin{theorem} [Minimal plurisubharmonic minorant]
\label{th:minpshminorant} 
Let $\omega$ be a domain in $\r^n$ $(n\ge 3)$, and let $\phi\colon \omega\to \r\cup\{-\infty\}$ 
be an upper semicontinuous function. Then the function 
\begin{equation}
\label{eq:minPoisson-funct}
   u(\bx ) := \inf \Big\{\int^{2\pi}_0 \phi(F(\E^{\imath t}))\, 
   \frac{dt}{2\pi} : \ F\in \Mgot(\d,\omega),\ F(0)=\bx \Big\}, \quad \bx \in \omega
\end{equation}
is minimal plurisubharmonic on $\omega$ or identically $-\infty$; 
moreover, $u$ is the supremum of the minimal plurisubharmonic 
functions on $\omega$ which are not greater than $\phi$. 
\end{theorem}

%
%
%
%
\begin{definition} \label{def:nullhull}
Let $K$ be a compact set in $\c^n$, $n\ge 3$. The {\em null hull}  of $K$ is the set
\begin{equation}\label{eq:nullhull}
		\wh K_{\Ngot} =\{\bz \in \c^n: v(\bz )\le \max_K v\ \ \ \forall v\in \NPsh(\c^n)\}.
\end{equation}  
\end{definition}

Note that  $\wh K_{\Ngot}$ is a special case of  a {\em $\mathbb G$-convex hull} introduced by 
Harvey and Lawson in \cite[Definition 4.3, p.\ 2434]{HarveyLawson2012AM}.
The maximum principle for subharmonic functions implies that
for any bounded null holomorphic curve $A\subset \c^n$ with boundary $bA\subset K$
we have $A\subset \wh K_{\Ngot}$. Since $\Psh(\c^n)\subset \NPsh(\c^n)$, we clearly have 
the inclusions
\begin{equation}\label{eq:comparehulls}
	K\subset   \wh K_{\Ngot}\subset  \wh K  \subset \Co(K).
\end{equation}
The polynomial hull $\wh K$ is  rarely equal to the convex hull $\Co(K)$, and
in general we also have $\wh K_{\Ngot}\ne \wh K$ (see \cite[Example 3.2]{DrinovecForstneric2015TAMS}).

The following characterization of the null hull agrees 
with \cite[Corollary 3.5]{DrinovecForstneric2015TAMS} 
for $n=3$. The proof in \cite{DrinovecForstneric2015TAMS}  
holds in any dimension $n\ge 3$, the nontrivial direction 
being furnished by Theorem \ref{th:nullpshminorant} in this paper. 
Recall that $|I|$ denotes the Lebesgue measure of a set $I\subset \r$.

%
%
%
%

\begin{corollary}\label{cor:DD}
Let $K$ be a compact set in $\c^n$ $(n\ge 3)$, and let  $\Omega\subset \c^n$ be a
bounded pseudoconvex Runge domain containing $K$.
A point $\bz\in \Omega$ belongs to the null hull $\widehat K_\Ngot$ of $K$ 
if and only if there exists a sequence of null discs $F_j\in \Ngot(\d,\Omega,\bz)$ 
such that 
\begin{equation}\label{eq:measure}
	\big| \{t\in [0,2\pi]: \dist(F_j(e^{\imath t}),K)<1/j\}  \big| \ge 2\pi -1/j, \quad 
	j=1,2,\ldots.
\end{equation}
\end{corollary}

%
%
%
%

Similarly, the following characterization of the minimal hull 
generalizes \cite[Corollary 4.9]{DrinovecForstneric2015TAMS} to any dimension $n\ge 3$.
The nontrivial direction is furnished by Theorem \ref{th:minpshminorant}.

\begin{corollary}\label{cor:minullhull}
Let $D$ be a minimally convex domain in $\r^n$ $(n\ge 3)$, let $K$ be a compact set in $D$, and
let $\omega \Subset D$ be a relatively compact domain containing the minimal hull $\widehat K_{\Mgot,D}$ of $K$.
A point $\bx \in \omega$ belongs to $\widehat K_{\Mgot,D}$ if and only if there exists a sequence of 
conformal minimal discs $F_j \colon \cd \to \omega$ such that, for every $j=1,2,\ldots$, we have $F_j(0)=\bx$ and
\eqref{eq:measure}.
\end{corollary}

\vspace{1mm}

%
%
%
%
\begin{remark} 
\label{rem:mean-hull}
Recall  (cf.\ Remark \ref{rem:mean-convex}) that a smoothly bounded domain 
$D\subset \r^n$ is mean-convex if and only if it is $(n-1)$-convex. 
By the maximum principle (see Proposition \ref{prop:max-principle}), 
mean-convex domains containing a given compact set $K\subset \r^n$
are natural barriers for minimal hypersurfaces with boundaries in $K$. 
The smallest such barrier, if it exists, is called the 
{\em mean-convex hull} of $K$; clearly it coincides with the $(n-1)$-convex hull $\wh K_{n-1}$ 
(see Definition \ref{def:p-hull}).  The main technique 
for finding the mean-convex hull is the {\em mean curvature flow} of hypersurfaces,  
introduced by Brakke \cite{Brakkebook}. For results on this subject we refer, among others, to the papers
\cite{MeeksYau1982MZ,GageHamilton1986JDG,EckerHuisken1989AM,EckerHuisken1991IM,ChenGigaGoto1991JDG,
HuiskenIlmanen2001JDG,HuiskenIlmanen2008JDG,Spadaro2011,MercierNovaga2015} 
and the monograph by Colding and Minicozzi \cite{ColdingMinicozzibook}.
Our proof of Corollary \ref{cor:minullhull} relies on completely different ideas,
but it applies only to the $2$-convex hull (which equals the 
mean-convex hull only in dimension $n=3$). We indicate the following natural question.

\begin{problem}
Let $K$ be a compact set with smooth boundary in $\r^n$ for some $n\ge 3$. Given a point $\bx_0\in \wh K_{\Mgot}$,
does there exist a conformal minimal disc $F\colon \cd \to \r^n$  such that $F(0)=\bx_0$ and $F(b\d)\subset K$?
\qed \end{problem}
\end{remark}

Recall that the {\em Green current} on $\c$ is defined on any $2$-form $\alpha = adx\wedge dy$ by
\[
	G(\alpha) = -\frac{1}{2\pi}  \int_\d \log|\cdotp|\, \alpha =
	- \frac{1}{2\pi}  \int_{\zeta \in \d} \log|\zeta|\cdotp a(\zeta) dx\wedge dy.
\]
Clearly, $G$ is a positive current of bidimension $(1,1)$ and $dd^c G=\sigma-\delta_0$,
where $\sigma$ is the normalized Lebesgue measure on the circle $\t=b\d$ and 
$\delta_\bz$ denotes the point mass at $\bz$.
If $F \colon \cd\to \c^n$ is a holomorphic disc, then $F_*G$ is a positive current of
bidimension $(1,1)$ on $\c^n$ satisfying  $dd^c (F_*G) = F_* \sigma - \delta_{F(0)}$. 
(See Duval and Sibony \cite[Example 4.9]{DuvalSibony1995}.) 

Assume now that $K$ is a compact set in $\c^n$, $\bz$ is a point in the null hull $\wh K_{\Ngot}$,
and $F_j\colon \cd\to\c^n$ is a sequence of holomorphic null discs with centers $F_j(0)=\bz$, 
furnished by Corollary \ref{cor:DD}. 
By Wold \cite{Wold2011JGA} (see also \cite[Proof of Theorem 6.2]{DrinovecForstneric2015TAMS}),
the sequence of Green currents $T_j=(F_j)_*G$ on $\c^n$ has a weakly convergent subsequence,
and the limit current $T$ satisfies $dd^c T=\mu-\delta_{\bz}$ where $\mu$ is a probability
measure on $K$. This generalizes the characterization of the null hull of a compact set in $\c^3$
by null positive Green currents,
given by \cite[Theorem 6.2]{DrinovecForstneric2015TAMS}, to any dimension $n\ge 3$.
Similarly, applying the above argument to the sequence of conformal minimal discs
$F_j\colon \cd\to \r^n$ furnished by Corollary \ref{cor:minullhull} and using the mass formula
in \cite[Lemma 5.1]{DrinovecForstneric2015TAMS}, we see that 
\cite[Theorem 6.4, Corollaries 6.5, 6.10]{DrinovecForstneric2015TAMS}
hold in any dimension $n\ge 3$, with the same proofs.

%
%
\begin{remark}
Recently, Sibony \cite{Sibony2017} found nonnegative directed currents of bidimension $(1,1)$
describing the $\Gamma$-hull $\wh K_\Gamma$ of a compact set $K\subset \c^n$
in any  directed system determined by a closed,
fiberwise conical subset $\Gamma$ of the tangent bundle $T\c^n$. 
The hull $\wh K_\Gamma$  is defined by the maximum principle 
in terms of $\Gamma$-plurisubharmonic functions; i.e., functions whose Levi form is nonnegative 
in directions from $\Gamma$. Sibony's characterization holds even if there are no $\Gamma$-directed holomorphic 
discs (i.e., discs whose derivatives lie in $\Gamma$); in particular, his $\Gamma$-directed current need not be limits of 
directed Green currents. The null hull falls within this framework; in this case, the fiber 
$\Gamma_z\subset T\c^n\cong \c^n$ over any point $z\in \c^n$ is 
the null quadric  \eqref{eq:Agot}, $\Gamma$-plurisubharmonic functions
are null plurisubharmonic functions, and $\Gamma$-discs are null discs.
(The classical case of the polynomial hull is due to Duval and Sibony \cite{DuvalSibony1995};
see also Wold \cite{Wold2011JGA}; in this case $\Gamma=T\c^n$.)
It seems an interesting question to decide in which systems directed by a complex 
analytic variety $\Gamma\subset T\c^n$ with conical fibers is it possible to describe the hull 
$\wh K_\Gamma$ by sequences of $\Gamma$-directed holomorphic discs $F_j\colon\cd\to\c^n$ 
whose boundaries converge to $K$ in measure (cf.\ \eqref{eq:measure}). 
For the polynomial hull, this holds by 
Poletsky \cite{Poletsky1991PSPM,Poletsky1993IUMJ} and Bu-Schachermayer \cite{BuSchachermayer1992TAMS}.
(For generalizations to complex manifolds, see 
\cite{LarussonSigurdsson1999IUMJ,LarussonSigurdsson2003JRAM,Rosay2003JKMS,Rosay2003IUMJ}; 
for complex spaces, see \cite{DrinovecForstneric2012IUMJ}.) 
For the null hull, this holds by \cite[Theorem 6.2]{DrinovecForstneric2015TAMS} 
(for $n=3$) and Corollary \ref{cor:DD} (for $n>3$). These seem to be the only cases studied so far.
\end{remark}


\subsection*{Acknowledgements}
A. Alarc\'on was supported by the Ram\'on y Cajal program of the Spanish Ministry of Economy and Competitiveness.
A.\ Alarc\'{o}n and F.\ J.\ L\'opez were partially supported by the MINECO/FEDER grants MTM2014-52368-P and MTM2017-89677-P, Spain. 
B.\ Drinovec Drnov\v sek and F.\ Forstneri\v c were partially  supported  by the research program P1-0291 and 
grants J1-5432 and J1-7256 from ARRS, Republic of Slovenia.


\vskip 0.3cm

\noindent Antonio Alarc\'{o}n

\noindent Departmento de Geometr\'ia y Topolog\'ia e Instituto de Matem\'aticas (IEMath-GR), Universidad de Granada, E--18071 Granada, Spain.

\noindent  e-mail: {\tt alarcon@ugr.es}

\vspace*{0.3cm}
\noindent Barbara Drinovec Drnov\v sek

\noindent Faculty of Mathematics and Physics, University of Ljubljana, and Institute
of Mathematics, Physics and Mechanics, Jadranska 19, SI--1000 Ljubljana, Slovenia.

\noindent e-mail: {\tt barbara.drinovec@fmf.uni-lj.si}

\vspace*{0.3cm}

\noindent Franc Forstneri\v c

\noindent Faculty of Mathematics and Physics, University of Ljubljana, and Institute
of Mathematics, Physics and Mechanics, Jadranska 19, SI--1000 Ljubljana, Slovenia.

\noindent e-mail: {\tt franc.forstneric@fmf.uni-lj.si}

\vspace*{0.3cm}

\noindent Francisco J.\ L\'opez

\noindent Departmento de Geometr\'ia y Topolog\'ia e Instituto de Matem\'aticas (IEMath-GR), Universidad de Granada, E--18071 Granada, Spain.

\noindent  e-mail: {\tt fjlopez@ugr.es}

\end{document}